\documentclass [12pt, makeidx]{amsart}
\pdfoutput=1
\usepackage{thmtools}
\usepackage{graphicx}
\usepackage{amsthm}
\usepackage{amssymb}
\usepackage[all]{xy}
\usepackage{xspace}
\usepackage{float}
\usepackage{enumitem}
\usepackage{url}
\usepackage{comment}
\usepackage{stmaryrd}
\usepackage{fourier}
\usepackage{xcolor}
\usepackage{xr-hyper}
\makeatletter
\newcommand*{\updatelabelname}[1]{%
  \xdef\@currentlabelname{#1}}
\makeatother

\newcommand{\hHom}{{\mathcal Hom}}

\newtheorem{thm}{Theorem}[section]
\newtheorem{cor}[thm]{Corollary}
\newtheorem{lem}[thm]{Lemma}

\newtheorem{prop}[thm]{Proposition}

\newtheorem{defn}[thm]{Definition}
\theoremstyle{definition}
\newtheorem{prop-defn}[thm]{Proposition and Definition}

\theoremstyle{remark}
\newtheorem{rem}[thm]{Remark}
\newtheorem{rems}[thm]{Remarks}

\newtheorem{conjecture}[thm]{Conjecture}
\newtheorem{Question}[thm]{Question}
\newtheorem*{Question*}{\bf Question}

\numberwithin{equation}{section}

\newcommand{\thistheoremname}{}
\newtheorem*{genericprop*}{\thistheoremname}
\newenvironment{namedprop*}[1]
  {\renewcommand{\thistheoremname}{#1}%
   \begin{genericprop*}}
  {\end{genericprop*}}

\newtheorem*{genericlem*}{\thistheoremname}
\newenvironment{namedlem*}[1]
  {\renewcommand{\thistheoremname}{#1}%
   \begin{genericlem*}}
  {\end{genericlem*}}
  
  \newtheorem*{genericthm*}{\thistheoremname}
\newenvironment{namedthm*}[1]
  {\renewcommand{\thistheoremname}{#1}%
   \begin{genericthm*}}
  {\end{genericthm*}}

  \newtheorem*{genericcond*}{\thistheoremname}
\newenvironment{namedcond*}[1]
  {\renewcommand{\thistheoremname}{#1}%
   \begin{genericcond*}}
  {\end{genericcond*}}

\newcommand{\Id}{{\rm Id}}

\newcommand{\Image}{{\rm Im}}

\newcommand{\codim}{{\rm codim}}
\newcommand{\Diff}{{\rm Diff}}
\newcommand{\vol}{{\rm vol}}
\newcommand{\osc}{{\rm osc}}
\usepackage{pdfsync}

\newcommand{\Mor}{{\rm Mor}}

\newcommand{\F}{{\mathcal F}}
\newcommand{\G}{{\mathcal G}}

\newcommand{\I}{{\mathcal I}}
\newcommand{\cI}{{\mathcal I}^{\bullet}}
\newcommand{\cH}{{\mathcal H}^{\bullet}}

\newcommand{\J}{{\mathcal J}}

\newcommand \id {{\rm id}}

\newcommand{\supp}{{\rm supp}}

\usepackage{scrtime}
\usepackage{calligra}

\makeindex

 \usepackage[mathscr]{eucal}
\newcommand {\gr}{\mathrm {gr}}

\newcommand {\diam}{\mathrm {diam}}

\newcommand {\GFQI}{G.F.Q.I.\xspace}

\usepackage{hyperref}
\DeclareMathAlphabet{\mathpzc}{OT1}{pzc}{m}{it}
\usepackage[percent]{overpic}
\usepackage{datetime}

\newcommand{\DHam }{\mathfrak{DHam}}

\newcommand {\LL}{{\mathfrak L}}
\newcommand{\comm}[1]{}
\makeindex
\usepackage[
  backend=biber,
     style=ext-alphabetic,
    maxbibnames=5,
    sorting=nyt,
    giveninits,
    articlein=false,
    url=true, 
    doi=true,
    eprint=true
]{biblatex}
\DeclareFieldFormat[article, incollection, unpublished]{title}{\sl{#1}}

\title[Inverse reduction inequalities for spectral numbers and applications]{Inverse reduction inequalities for spectral numbers and applications}
\author{C. Viterbo }
\thanks{Université Paris-Saclay, CNRS, Laboratoire de mathématiques d’Orsay, 91405, Orsay, France. Part of this paper was written as the author was a member of DMA, \'Ecole Normale Sup\'erieure, 45 Rue d'Ulm, 75230 Cedex 05, FRANCE. We also acknowledge support from ANR MICROLOCAL (ANR-15-CE40-0007) and COSY (ANR-21-CE40-0002)}	

\begin{document}
\def \Z {\mathbb Z}
\def \H {\mathcal H}
\def \cF {\F^{\bullet}}
\def \cG {\G^{\bullet}}
\def \cI {\I^{\bullet}}
\def \cJ {\J^{\bullet}}
\def \Char {{\rm Char}}
\def \card {{\rm card}}
\def \cstar {\varoast}
\begin{abstract}
Our main result is the proof of an inequality between the spectral numbers of a Lagrangian and the spectral numbers of its reductions, going in the opposite direction to
the classical inequality (see e.g \cite{Viterbo-STAGGF}). This has applications to the "Geometrically bounded Lagrangians are spectrally bounded" conjecture from \cite{SHT} and to the structure of elements in the $\gamma$-completion of the set of exact Lagrangians (see \cite{Viterbo-gammas}). We also investigate the local path-connectedness of the set of Hamiltonian diffeomorphisms with the spectral metric. 
\end{abstract}

\maketitle
\today ,\;\; \currenttime

\tableofcontents
\section{Introduction}
The goal of this paper is twofold. We first prove an ``inverse reduction inequality" for spectral capacities. Classically (see \cite{Viterbo-STAGGF} or Proposition \ref{Prop-reduction-inequality}) the spectral norm of an exact  Lagrangian is bounded from below by the supremum of the spectral norms of its reductions. Here we prove that with proper assumptions, and up to a constant factor depending only on the dimension, the spectral norm of the Lagrangian can be also bounded from above by the supremum of the spectral norms of its reductions. This is Theorem \ref{Thm-Reverse-reduction-inequality}.

A consequence of this inverse reduction inequality is that the following property holds in the  cotangent bundles $T^*N$ of certain manifolds (they are specified in Section \ref{Section-6}): if $L$ is an exact Lagrangian far away from the zero section, then we may move $L$ far away from itself by a Hamiltonian flow preserving the zero section. This is Theorem \ref{Thm-6.3}. 

 This result is used in \cite{Viterbo-gammas} to prove that for certain elements in the $\gamma$-completion (or Humili\`ere completion) of $\mathcal L(T^{*}N)$, the $\gamma$-support (defined in \cite{Viterbo-gammas}) determines the element.  

The second goal of the paper is an application of these results to a conjecture that was stated and used in a preliminary version of \cite{SHT} (the present version does not use this result), according to which if $L$ is contained in $DT^*N=\{ (q,p)\in T^*N \mid \vert p \vert \leq 1\}$
then $\gamma (L) \leq C_N$ for some constant $C_N$. 

This conjecture was first proved by Shelukhin in \cite{Shelukhin-Zoll, shelukhin-sc-viterbo} for a certain class of manifolds (called ``string point-invertible'') and for Zoll symmetric spaces. 

We prove this conjecture for another  class in Theorem \ref{Thm-viterbo-conjecture}. Our class is different from Shelukhin's, but it is not clear whether one contains the other. As we shall see our class contains all (compact) homogeneous spaces, in particular the spheres of any dimensions, the projective spaces,  (according to \cite{shelukhin-sc-viterbo} spheres of even dimensions and projective spaces are not ``string-point invertible", so are not covered\footnote{but are covered by those of \cite{Shelukhin-Zoll}.} by the results of \cite{shelukhin-sc-viterbo}). 

Another proof, valid for Lie groups and homogeneous spaces is due to Guillermou and Vichery (see \cite{Guillermou-Vichery}) was obtained independently from the one presented here. 

In a last section we discuss the issue of local path-connectedness of the space of Hamiltonian diffeomorphisms for the metric $\gamma$. 

Returning to the inverse reduction inequality, note that having both a reduction inequality and an inverse reduction inequality tells something about how spectral norms measure the size of objects. By ''size'' we here mean a map
$\sigma: \mathcal S_\sigma \longrightarrow [0,+\infty]$, where $\mathcal S_\sigma$ is a subset of $\mathcal P (X)$, the set of subsets of $X$, such that for $A \subset B$ we have $\sigma (A)\leq \sigma(B)$.

For example the diameter of a set, or its Lebesgue measure are sizes, with $\mathcal S_\sigma = \mathcal P(X)$ in the first case, $\mathcal S_\sigma $ is the set of measurable sets in the second case. There are many more examples, from integral geometry like the width (minimal distance between parallel hyperplanes containing $A$),  harmonic analysis - like the Newtonian capacity, symplectic geometry- like Gromov's width, the Ekeland-Hofer or Hofer-Zehnder capacities, the displacement energy, the  spectral capacity, etc...

Now if we think of reductions as ``slices" of a set, we can think of  sizes that are  bounded from below by the supremum of the sizes of its slices in a fixed direction :  for example the diameter of a set can be estimated from below by the supremum of the diameter of its slices, but not from above :  if $H$ is a hyperplane, 
$$ \sup_{x\in {\mathbb R}^n} \diam (A\cap (x+H)) \leq \diam (A)$$ but if $A$ is a segment orthogonal to $H$ the left hand side vanishes while $\diam(A)$ can be arbitrarily large, so there cannot be an estimate in the opposite direction. 
On the other hand 
the volume of a body - let's say contained in a fixed cube- can be estimated from above by the supremum of the volumes of its slices in a fixed direction (by Fubini's formula) :
$$\vol_n(A) = \int \vol_{n-1}(A\cap (x+H))dx \leq C \sup_{x}\vol_{n-1}(A\cap (x+H))$$

But it cannot be estimated from below by such a  supremum : a set of small volume concentrated near a hyperplane can have a slice with very large volume. 
We notice that an estimate from above was  already discovered in \cite{Viterbo-isoperimetric} for the displacement energy (defined in \cite{Hofer-distance}) rather than the spectral norm. However an estimate from below does not hold in this case... 

Remarkably, the spectral norm satisfies both inequalities. 

\section{Comments and acknowledgements}
This paper is strongly related to  \cite{Viterbo-gammas}, where Theorem \ref{Thm-6.3} is crucially used. 
The present paper was originally written using generating functions as the main tool, and assumed the Lagrangians were Hamiltonialy isotopic to the zero section. In order to extend our results to the general exact case, we ``translated'' our paper to this more general setting. In this article we sometimes repeat the argument using generating functions as this may clarify the ideas involved. 
Finally I wish to thank Stéphane Guillermou for his careful reading of the paper and for pointing out some inaccuracies. 

\section{Basic definitions and notations}
We consider the set of exact embedded Lagrangians in $T^*N$ and denote it by $\mathfrak L (T^*N)$. Since any two exact Lagrangian manifolds intersect (see\footnote{In fact this is more or less explicit in \cite{Gromov-symplectic}, $2.3B'_{3}$. The authors we mention prove much more.
} \cite{Fukaya-Seidel-Smith, Kragh-Abouzaid}), such a manifold must be connected. We often need to consider the pairs $(L, f_L)$ where $f_L$ is a primitive of $\lambda_{\mid L}$ on $L$. Note for a fixed element $L \in \mathfrak L (T^*N)$ there is a one-dimensional family of elements of pairs $(L,f_L)$ obtained by adding a constant to $f_L$. We denote by $\widetilde L +c$ the element associated to $f_L+c$. 

 For a symplectic manifold $(M,\omega)$ we denote by $\Lambda (M)$ the bundle with fiber over $z\in M$ given by $\Lambda (T_zM)$, the Grassmannian Lagrangian of  $T_zM$,  and by $\widetilde \Lambda (M)$ the bundle over $M$  with fiber the universal cover of $\Lambda (T_zM)$.

The Gauss map $G_L$ sends $L$ to a section over $L$ of $\Lambda (T^*N)$ and we call ``grading'' of $L$ a lift  $\widetilde G_L$ of $G_L$ to a section of $\widetilde \Lambda(T^*N)$ over $L$ (see \cite{Seidel-graded}). Since  by the Kragh-Abouzaid theorem (see \cite{Kragh-Abouzaid}) the Maslov class of $L$ always vanishes, such a lift always exists. 

We define $\mathcal L(T^*N)$ as the set of triples $\widetilde L= (L, f_L, \widetilde G_L)$. The obvious action of $\pi_1(\Lambda(n))=\mathbb Z$  on 
 $\widetilde \Lambda(T^*N)$ induces an action denoted by $T$ on $\mathcal L(T^*N)$. This action is also denoted by $T^k(\widetilde L)=\widetilde L [k]$.

Note that in $T^*N$, the bundle $ \widetilde \Lambda (T^*N)$ can be constructed explicitly as follows : choose $J$ an almost complex structure such that $JT_{(x,0)}0_N=V_{(x,0)}$ where $V_{(x,p)}$ is the tangent space to the fiber of the projection $T^*N \longrightarrow N$ at $(x,p)$.   Then $\widetilde \Lambda_{(x,p)} (T^*N)$ is the set of continuous paths from $JV_{(x,p)}$ to $T \in \Lambda (T^*N)$ modulo the relation of homotopy with fixed endpoints. 

We may thus speak for $L=0_N$ of the ``trivial lift'' corresponding to the constant path at $T_{(x,0)}0_N$ and by  abuse of notation, we write $0_N$ for the zero section with $f_L\equiv 0$ and this trivial lift. 
 
  Moreover given a Hamiltonian $H(t,z)$ on $T^*N$, and $\widetilde L_0 \in \mathcal L(T^*N)$ we can define $\varphi_H(\widetilde L_0) \in \mathcal L(T^*N)$. 
 The function $f_{L_1}$ is given by 
 $$f_{L_1}(z_1)= \varphi_H(f_{L_0})(z_1)=\int_0^1 p\dot q - H(t,q,p) dt + f_L(z_0)$$ where $(q(t),p(t))=\varphi_H^t(z_0)=z_t$  (see proposition 5.2 in \cite{Viterbo-gammas}). 
Note that this depends on the choice of $H$ and not only on the flow (so we commit an abuse of notation here). Similarly given $\widetilde G_{L_0}$ a lift of the Gauss map of $L_0$, $\varphi_H$ defines a lift $\widetilde G_{L_1}$ of the Gauss map of $L_1$  In particular we can replace $H$ by $H(t,z)+c(t)$ and $f_{L_1}$ will be replaced by  $f_{L_1}+c(1)-c(0)$. 
Of course the projection of $\mathcal L(T^*N)$ on $\mathfrak L(T^*N)$ is a fibration with fiber $\mathbb R \times \mathbb Z$ (the first factor corresponds to the $f_L+c$ the second one to $T^k(\widetilde L)$). 

 Note that we shall, whenever we feel that this causes no confusion, just write $L$ instead of $\widetilde L$ or $(L, f_L, \widetilde G_L)$. 
 
\section{Spectral invariants for sheaves and Lagrangians}\label{Section-4}

We denote by $D^b(X)$ the derived category of bounded complexes of sheaves on $X$. For $\cF, \cG \in D^b(X \times {\mathbb R} )$ the operation  $\cF \cstar \cG$ is defined as follows. Let $s: N\times {\mathbb R}\times N \times {\mathbb R} \longrightarrow N\times N \times {\mathbb R}  $ be the map given by $s(x_1,t_1, x_2,t_2)=(x_1,x_2,t_1+t_2)$
and $d: N\times {\mathbb R} \longrightarrow N\times N \times {\mathbb R} $ given by $d(x,t)=(x,x,t)$. We set  
$$ \cF \cstar \cG = (Rs)_! d^{-1}(\cF \boxtimes \cG)$$
and $R\hHom^\cstar$ is the adjoint of $\cstar$ in the sense that 
$$\Mor_{D^b(X\times {\mathbb R} )} (\cF, R\hHom^\cstar(\cG, \cH))=\Mor_{D^b(X\times {\mathbb R})} (\cF \cstar \cG, \cH)$$

Let $\widetilde L=(L,f_L, \widetilde G_L) \in \mathcal L(T^*N)$  and $$\widehat L=\left\{(q,\tau p, f_L(q,p), \tau) \mid (q,p)\in L, \tau > 0\right \}$$
the homogenized Lagrangian in $T^*(N\times {\mathbb R})$. 

For a sheaf $\cF$ in $D^{b}(N)$ we define $SS(\cF)$ as the singular support of $\cF$ (see \cite{K-S}, proposition 5.1.1, p.218) and denote by $SS^{\bullet}(\cF)$ the same set with the zero section removed, that is $SS(\cF)\setminus 0_{N}$. 
According to Guillermou and the author (see \cite{Guillermou} for \ref{TQ1}), (\ref{TQ2}), (\ref{TQ4}) and  \cite{Viterbo-Sheaves} for the other properties),  we have

\begin{thm} \label{Thm-Quantization}
To each $\widetilde L\in \mathcal L(T^*N)$ we can associate $\cF_{L} \in D^b(N)$ such that 
\begin{enumerate} 
\item\label{TQ1} $SS^{\bullet}(\cF_{ L})=\widehat L$
\item\label{TQ2} $\cF_{\widetilde L}$ is pure (cf. \cite{K-S} page 309), $\F_L=0$ near $N\times \{-\infty\}$ and $\F_L=k_N$ near $N\times \{+\infty\}$
\item\label{TQ3} We have an isomorphism
$$FH^\bullet(L_0,L_1;a,b)=H^*\left (N\times [a,b[, R\hHom^{\cstar}(\cF_{L_0},\cF_{L_1})\right)$$
\item \label{TQ4}  $\cF_L$ is the unique element in $D^b(X\times {\mathbb R})$  satisfying properties (\ref{TQ1}) and (\ref{TQ2}). 
\item \label{TQ5} There is a natural product map  $$
 R\hHom^\cstar(\cF_{L_1},\cF_{L_2}) \otimes R\hHom^\cstar(\cF_{L_2},\cF_{L_3}) \longrightarrow  R\hHom^\cstar(\cF_{L_1},\cF_{L_3})
$$
inducing in cohomology a map
\begin{gather*}
H^*(N\times [\lambda , +\infty [, R\hHom^{\cstar}(\cF_{L_1},\cF_{L_2})) \otimes H^*(N\times [\mu , +\infty [, R\hHom^{\cstar}(\cF_{L_2},\cF_{L_3}))
\\  \Big\downarrow \cup_{\cstar} \\ H^*(N\times [\lambda + \mu , +\infty [, R\hHom^{\cstar}(\cF_{L_1},\cF_{L_3}))
\end{gather*}
that coincides through the above identifications to the triangle product in Floer cohomology.\end{enumerate} 
\end{thm} 
\begin{rem} 
The grading $\widetilde G_L$ defines the grading of $\cF$, hence of the Floer cohomology. 
\end{rem} 
Note that for $X$ open, we denoted by $H^*(X\times [\lambda, \mu[, \cF)$  the relative cohomology of sections on $X \times ]-\infty, \mu[$ vanishing on $X \times ]-\infty, \lambda[$ and fitting in the exact sequence
$$
H^*(X\times [\lambda, \mu[, \cF) \longrightarrow H^*(X\times [-\infty, \mu[, \cF) \longrightarrow H^*(X\times [-\infty, \lambda[, \cF)
$$
It is also equal to the cohomology associated to the derived functor $R\Gamma_Z$ where $Z$ is the locally closed set $X\times [\lambda, \mu[$.
We should write $\cF_{\widetilde L}$ instead of $\cF_{L}$ but this abuse of notation should be harmless. 
We shall now define the {\bf pseudo-Tamarkin category}. The standard definition for the Tamarkin category is the left-orthogonal\footnote{An object $X$ is left-orthogonal to $Y$ if  $\Mor(X,Y)=0$. If $\mathcal D$ is a full triangulated subcategory of $\mathcal  C$, then $\mathcal D^{\perp}$ is the full triangulated subcategory with objects $\left\{X \in \mathcal C \mid \Mor(X,Y)=0 \forall Y \in \mathcal D\right \}$. } of the category of sheaves such that $SS(\cF)\subset \{\tau \leq 0\}$. This is contained in the pseudo-Tamarkin category, but the pseudo-Tamarkin is easier for our purposes, since morphisms are again objects in the category (this is not the case in the Tamarkin category). 

\begin{defn}[Pseudo-Tamarkin category]\label{Def-Tamarkin}
We denote by $\mathcal T_0(N)$ the set of 
 $\cF$ in $D^b(X \times {\mathbb R} )$ such that 
 \begin{enumerate} 
 \item  $SS(\cF) \subset \{\tau \geq 0\}$ 
 \item   $\F_L=0$ near $N\times \{-\infty\}$ 
 \item  $\F_L=k_N$ near $N\times \{+\infty\}$
 \end{enumerate} 
\end{defn} 

\begin{defn}[see \cite{Vichery-these}, Section 8.3]
Let $\cF$be an element in $\mathcal T_0(X)$. 
Let $\alpha \in H^*(N \times {\mathbb R}, \cF)\simeq H^*(N)$ be a nonzero class. We define 
$$c(\alpha, \cF)= \sup \left\{ t \in {\mathbb R}  \mid \alpha \in \Image ( H^*(N \times [t, +\infty[, \cF))\right \}$$
\end{defn} 

Note that $\cF_L$ satisfies Property (\ref{TQ2}) of Theorem \ref{Thm-Quantization}, so $H^*(N \times {\mathbb R}, \cF)\simeq H^*(N)$ and thus we have, using the canonical map
$$H^*(N \times [t, +\infty[, \cF) \longrightarrow H^*(N \times  {\mathbb R}, \cF)$$
and Theorem \ref{Thm-Quantization}, (\ref{TQ3}) the following

\begin{cor} Let $\widetilde L $ be an element in $\mathcal L(T^{*}N)$. 
Let $\alpha \in H^*(N \times {\mathbb R}, \cF)\simeq H^*(N)$ be a nonzero class. 
Then $c(\alpha, \cF_L)$ coincides with the spectral invariant $c(\alpha,L)$ associated to $\alpha$ using Floer cohomology. 
\end{cor}  
As a consequence the $c(\alpha, \cF_L)$ satisfies the properties of the Floer homology Lagrangian spectral invariants, and in particular the triangle inequality, since this holds in Floer homology (see \cite{Hum-Lec-Sey3}, theorem 17). However we shall sometimes need to extend the triangle  inequality to situations where $\cF$ is in $\mathcal T_0(X)$ but does not necessarily  correspond to an exact embedded Lagrangian.   

\begin{prop} [Triangle inequality in the pseudo-Tamarkin category-see \cite{Vichery-these}, proposition 8.13]\label{Prop-triangle}
Let  $\cF_1, \cF_2, \cF_3$ be complexes of sheaves in $\mathcal T_0(N)$. 
Then we have 
\begin{enumerate} 
\item 
$$ c(\alpha \cup \beta; \cF_1,\cF_3) \geq c(\alpha ; \cF_1,\cF_2) + c( \beta; \cF_2,\cF_3)$$
\item $$c(1, \cF_1,\cF_2)=- c(\mu;\cF_2,\cF_1)$$
\end{enumerate} 
\end{prop} 
\begin{proof} 
We set $\cF_{i,j}= R\hHom^{\cstar}(\cF_{L_i},\cF_{L_j})$ and we have, according to \cite{Viterbo-Sheaves} a product $$\cup_{\cstar} :\cF_{1,2} \otimes \cF_{2,3} \longrightarrow \cF_{1,3}$$ inducing the cup-product 
$$H^*( X\times [s,+\infty[ ; \cF_{1,2}) \otimes H^*( X\times [t,+\infty[ ; \cF_{2,3}) \longrightarrow H^*( X\times [s+t,+\infty[ ; \cF_{1,3})$$
Then we have the diagram

\xymatrix{
H^*( X\times [s,+\infty[ ; \cF_{1,2}) \otimes H^*( X\times [t,+\infty[ ; \cF_{2,3})\ar[d] \ar[r]^-{\cup_{\cstar}}& H^*( X\times [s+t,+\infty[ ; \cF_{1,3}) \ar[d]\\
H^*( X\times  {\mathbb R}  ; \cF_{1,2}) \otimes H^*( X\times {\mathbb R}  ; \cF_{2,3}) \ar[r]^-{\cup_{\cstar}}& H^*( X\times {\mathbb R}  ; \cF_{1,3})
}
where horizontal arrows are cup-products and vertical arrows restriction maps. So if $\alpha\otimes \beta$ is in the image of the left-hand side vertical arrow, which is equivalent to $s\leq c(\alpha, \cF_{1,2}), t \leq  c(\beta, \cF_{2,3})$, we have $\alpha\cup \beta$ is in the image of the right hand side, so that
$ s+t\leq c(\alpha\cup \beta , \cF_{1,3})$. This proves our claim. 
 \end{proof} 
 
 Let $\mu_N\in H^n(N)$ be the fundamental class of $N$ and $1_N\in H^0(N)$ the degree $0$ class. 
 \begin{defn} We set for $\cF$ in $\mathcal T_0(N)$
 $$c_+(\cF)=c(\mu_N, \cF)$$
 $$c_-(\cF)=c(1_N, \cF)$$
 $$\gamma (\cF)=c_+(\cF)-c_-(\cF)$$
 We set $\mathbb D \cF$ to be the Verdier dual of $\cF$ and $s(x,t)=(x,-t)$ and $\check{\cF}$ is quasi-isomorphic to  $0 \to k_{N\times {\mathbb R} } \to s^{-1}(\mathbb D \cF)\to 0 $.
 \end{defn} 
 We notice that $SS(\mathbb D \cF)=-SS(\cF)$ where for $A \subset T^*(N\times {\mathbb R})$, we set $-A=\{(x,-p, t, -\tau) \mid (x,p,t,\tau)\in A\}$ (see \cite{K-S} Exercise V.13, p. 247). As a result,  $\check{\cF}_L=\cF_{-L}$ where $-L=\{(q,-p) \mid (q,p)\in L\}$. The triangle inequality then implies
 \begin{prop} \label{Prop-3.6}
 We have for $\cF$ constructible in $\mathcal T_0(N)$
  \begin{enumerate} 
 \item \label{Prop-3.6-i} $c_+(\cF) \geq c_-(\cF)$ 
 \item \label{Prop-3.6-ii}
 $c_+(\cF)=-c_-(\check{\cF})$ 
 $c_-(\cF)=-c_+(\check{\cF})$ 
 so that $\cF \longrightarrow \check{\cF}$ is an $\gamma$-isometry. 
 \end{enumerate} 
 And of course if $L\in \mathcal L(T^*N)$ we have  
 $c_\pm(\cF_L)=c_\pm(L)$ and $\gamma(\cF_L)=\gamma(L)$.
 \end{prop} 
 \begin{proof} 
 Note that for $\cF$ cohomologically  constructible, we have $\check{\check \cF}=\cF$
 
 \end{proof} 
 Of course if $L\in \mathfrak L (T^*N)$ the $c_\pm(L)$ are not well defined (they are only defined up to constant), however $\gamma(L)$ is well-defined. 
\section{The inverse reduction inequality and applications}\label{Section-7}
The following inequality was proved in \cite{Viterbo-STAGGF} using generating functions (so for Lagrangians Hamiltonianly isotopic to the zero section), but the same proof holds for general exact Lagrangians, using sheaves
\begin{prop}[Reduction inequality, see \cite{Viterbo-STAGGF}, Prop. 5.1]\label{Prop-reduction-inequality}
Let $\widetilde L_1,\widetilde L_2$ be elements in $\mathcal L (T^*(X\times Y))$ and $(\widetilde L_1)_x, (\widetilde L_2)_x$ be their reduction by $T_x^*X\times T^*Y$, assumed to be in
$\mathcal L (T^*Y)$. We then have
$$\sup_{x\in X} c_+ ((\widetilde L_1)_x,(\widetilde L_2)_x) \leq c_+ (\widetilde L_1;\widetilde L_2))$$
$$c_- (\widetilde L_1,\widetilde L_2) \leq \inf_{x\in X}c_-((\widetilde L_1)_x,(\widetilde L_2)_x)$$
and as a consequence
$$\sup_{x\in X} \gamma ((\widetilde L_1)_x,(\widetilde L_2)_x) \leq \gamma (\widetilde L_1,\widetilde L_2)$$
\end{prop} 
In our setting, this will follow immediately from 
\begin{prop} \label{Prop-4.2}
Let $\cF$ in $\mathcal T_0(X)$. 
Then setting $\cF_V$ to be the restriction of $\cF$ to the closed submanifold $V$ of $X$ we have
$$c_-(\cF) \leq c_-(\cF_V)$$
$$c_+(\cF) \geq c_+(\cF_V)$$
hence 
$$ \gamma(\cF_V)\leq \gamma(\cF) $$
In particular if $\widetilde L_V \in \mathcal L(T^*V)$ is the reduction of $\widetilde L \in \mathcal L(T^*X)$ we have
$$c_-(\widetilde L) \leq c_-(\widetilde L_V)$$
$$c_+(\widetilde L) \geq c_+(\widetilde L_V)$$
hence 
$$ \gamma(\widetilde L_V)\leq \gamma(\widetilde L) $$

\end{prop} 
\begin{proof} 
The map 
\begin{gather*} H^0(X) \simeq H^0(X\times {\mathbb R} ; \cF) \longrightarrow H^0(V\times {\mathbb R} ; \cF)=H^0(V\times {\mathbb R} ,\cF)\simeq H^0(V) 
\end{gather*} 
coincides with the map   $H^0(X) \longrightarrow H^0(V)$, induced by the inclusion. This maps sends  $1_{X}$ to $1_V$ and we thus get the first inequality. The second one follows from the equality 
$c_+(\cF_L)=-c_-(\check{\cF_L})$. \end{proof} 
\begin{cor} \label{Cor-4.3}
Let $\cF$ in $\mathcal T_0(X\times Y)$.
Then setting $\cF_x$ to be the restriction of $\cF$ to the closed submanifold $\{x\} \times Y$, we have
$$c_-(\cF) \leq \inf_{x\in X}c_-(\cF_x)$$
$$c_+(\cF) \geq \sup_{x\in X}c_+(\cF_x)$$
hence 
$$ \sup_{x\in X} \gamma(\cF_x)\leq \gamma(\cF) $$
\end{cor} 
\begin{proof} 
This follows from applying Proposition \ref{Prop-4.2} to the family of submanifolds $V=\{x\}\times Y$. The last one of course follows from the inequality
\begin{gather*} \sup_x \gamma (\cF_x) = \sup_x( c_+ (\cF_x)-c_-(\cF_x)) \leq \\ \sup_xc_+(\cF_x) - \inf_x c_-(\cF_x) \leq c_+(\cF) - c_-(\cF)=\gamma (\cF)
\end{gather*} 

\end{proof} 
\begin{proof} [Proof of Proposition \ref{Prop-reduction-inequality}]
Let $\cF_j=\cF_{L_j}$ in $D^b(N \times {\mathbb R} )$  be the quantizations of the $\widetilde L_j$. Then  $\cF_{j,x}$ restriction of $\cF_j$ to $\{x\}\times Y$ defines $(\widetilde L_j)_x$ the reduction of $L_j$. Our result  is then equivalent to 
$$ \sup_{x\in X} c_+ ((\cF_1)_x,(\cF_2)_x) \leq c_+ (\cF_1;\cF_2))$$
Since $c_+ (\cF_1;\cF_2))=-c_- (\cF_2;\cF_1))$ setting 
$\cF=R\hHom^\cstar(\cF_1, \cF_2)$
this is equivalent to 
$$c_-(\cF) \leq \inf _{x\in X} c_- (\cF_x)$$
Since  $\cF$ is in $\mathcal T_0(X\times Y)$, we have by Corollary \ref{Cor-4.3}, that
 for all $x\in X$
$$  c_+ (\cF_x) \leq c_+ (\cF)$$ and
$$c_-(\cF) \leq c_- (\cF_x)$$
which proves our Proposition. 
\end{proof} 
Note that Proposition \ref{Prop-4.2} implies the following: let $f$ be a smooth map from $X$ to $Y$ and $\Lambda_f$ be the Lagrangian correspondence in $T^*(X\times Y)$ $$\Lambda_f= \{ (x,p_x,y,p_y)\in T^*X\times T^*Y\mid y=f(x),  p_x=p_y\circ df(x)\}$$  (i.e. $\Lambda_f$ is a Lagrangian in $\overline {T^*X}\times T^*Y$, note that it is homogeneous, so we may set the primitive of the Liouville form to be $0$). 

For $\widetilde L\in \mathcal L(T^*X)$ we denote by $\Lambda_f\widetilde L$ the image of $\widetilde L$ by $\Lambda_f$ in $T^*Y$, defined by 
$$\Lambda_f\widetilde L= \{(y,p_y) \mid \exists (x,p_x)\in L, (x,p_x,y,p_y)\in \Lambda_f\}$$
Conversely for $\tilde L\in \mathcal L(T^*Y)$ we have
$$\Lambda_f^{-1} \widetilde L= \{(x,p_x) \mid \exists (y,p_y) \in L, y=f(x), p_x=p_y\circ df(x)\}$$

We will assume  $f$ is non-characteristic for $\widehat L$ or equivalently $\Lambda_{f}$ is transverse to $0_{X}\times \widehat L$ away from the zero section (note that the sets are homogeneous). This is a generic property, and a small perturbation of $L$ or $f$ makes $f$ non-characteristic for $L$. 
\begin{rem} 
It is easy to see that if $f$ is a submersion (i.e. $df(x)$ is surjective for all $x\in X$), then $\Lambda_{f}^{-1}\widetilde L$ is always embedded, hence $\Lambda_{f}^{-1}$ defines a map from $\mathcal L (T^{*}Y)$ to $\mathcal L (T^{*}X)$ 
\end{rem} 
As a result $\Lambda_f^{-1}L$ is Lagrangian and exact  (and of course if $L$ is homogeneous, so is $\Lambda_f^{-1}L$). 
We claim
\begin{prop} \label{Prop-4.4}
Let $\cF \in \mathcal T_0(Y)$ be a constructible sheaf, and $f\in C^0(X,Y)$, then we have for $\alpha \in H^*(X), \beta \in H^*(Y)$
$$c(\beta, \cF) \leq c(f^*(\beta), f^{-1}\cF)$$
$$c(\alpha, f^{!}\cF) \leq c(f_!(\alpha),  \cF)$$
As a result, applying the above to $\cF_L$, we have for $\widetilde L\in \mathcal L(T^*Y)$ and $f$ a smooth map from $X$ to $Y$ such that $\Lambda_{f}$ is transverse (away from the zero section) to $L$ (i.e. $L$ is non-characteristic for $f$), then  $SS(f^{-1}(\cF))=\Lambda_{f}^{-1}\circ L$ and for  $\alpha \in H^*(X), \beta \in H^*(Y)$
$$c(\beta, \widetilde L) \leq c(f^*(\beta), \Lambda_f^{-1}\circ \widetilde L)$$
$$c(\alpha, \Lambda_f^{-1}\widetilde L) \leq c(f_!(\alpha),  \widetilde L)$$
\end{prop} 
\begin{proof} The proof is a sheafification of Proposition 7.38 in \cite{Viterbo-book}. We denote by $\tilde f$ the map $f\times \Id_{ {\mathbb R} }$ from $X \times {\mathbb R} $ to $Y \times {\mathbb R} $.  The natural transformation $ \id \longrightarrow (R\tilde f)_*\tilde 
f^{-1}$,  yields a map 
$$H^*(Y\times ]-\infty, c[, \cF) \longrightarrow H^*(Y\times ]-\infty, c[, (R\tilde f)_*\tilde f^{-1}(\cF))$$ but the right-hand side equals $H^*(X\times ]-\infty, c[, \tilde f^{-1}(\cF))$, so we have a morphism
$$f^{\#}: H^*(Y\times ]-\infty, c[, \cF) \longrightarrow  H^*(X\times ]-\infty, c[, \tilde f^{-1}(\cF))$$  
(for constant sheaves it corresponds to the pull-back $f^*$ in ordinary cohomology) hence a commutative diagram
 \begin{center}
\leavevmode
\xymatrix{
H^*(Y\times  {\mathbb R} , \cF)  \ar[d]^{\tilde f^{\#}}\ar[r] & H^*(Y\times ]-\infty, c[, \cF) \ar[d]^{\tilde f^{\#}} \\
H^*(X\times {\mathbb R} , \tilde f^{-1}\cF)  \ar[r] & H^*(X\times ]-\infty, c[, \tilde f^{-1}\cF) \\
}
\end{center}
For $\beta \in H^*(Y)= H^*(Y\times  {\mathbb R} , \cF_L)$, if $c \leq c(\beta, \cF)$ the image of $\beta$ in $H^*(Y\times ]-\infty, c[, \cF_L)$ vanishes,  i.e. $c \leq c(\beta, \widetilde  L)$ then the image of $f^*(\beta)$ in $H^*(X\times ]-\infty, c[, \tilde f^{-1}\cF_L)$ vanishes, i.e. $c \leq c(\beta, \Lambda_f^{-1}\widetilde L)$. Thus we get the first inequality.

For the second one, we similarly have a natural transformation $Rf_!f^! \longrightarrow \Id$ so that we have a map
$$H^*(Y\times ]-\infty, c[, R\tilde f_!\tilde f^!\cF) \longrightarrow  H^*(Y\times ]-\infty, c[, \cF)$$
hence 
$$\tilde f_{\#}: H^*(X\times ]-\infty, c[, \tilde f^!\cF) \longrightarrow  H^*(Y\times ]-\infty, c[, \cF)$$
since here $X,Y$ are compact. 

But we thus get the
 \begin{center}
\leavevmode
\xymatrix{
H^*(Y\times  {\mathbb R} , \cF)  \ar[r] & H^*(Y\times ]-\infty, c[, \cF) \\
H^*(X\times {\mathbb R} , \tilde f^!\cF) \ar[u]^{\tilde f_{\#}} \ar[r] & H^*(X\times ]-\infty, c[, \tilde f^{!}\cF)\ar[u]^{\tilde f_{\#}} \\
}
\end{center}
and taking $\alpha \in H^*(X)\simeq H^*(X\times {\mathbb R} , \cF)$, and using the fact that $f_{\#}$ coincides with the usual umkehr map $f_!$ at $+\infty$ (i.e. for $\cF=k_X$), 
we see that 
$$c(\alpha, f^!\cF) \leq c(f_!(\alpha), \cF)$$

Now if   $\cF_L$ is the quantization of $\widetilde L$, then ${\tilde f}^{-1}(\cF_L)$ and $\tilde f^!\cF_L$ have their singular supports, 
$SS({\tilde f}^{-1}(\cF_L))$ and $SS({\tilde f}^{-1}(\cF_L))$  both contained in $\Lambda_f^{-1}(SS(\cF_L)=\Lambda_f^{-1}L$ (see corollary 6.4.4, p. 270 in \cite{K-S}). Since $\Lambda_f^{-1}L$ is a smooth Lagrangian, it contains no proper coisotropic set (see Proposition 9.13 in \cite{Guillermou-Viterbo}), and thus the inclusions are equalities, and  we get the inequalities in  second statement.
\end{proof} 
\begin{cor}\label{Cor-5.5}
If $\cF$ is constructible and belongs to $\mathcal T_0$ we have
\begin{enumerate} 
\item \label{Cor-5.5-i} 
$$\gamma (f^{-1}\cF)\leq \gamma(\cF)$$
  \item\label{Cor-5.5-ii} If $f:X \longrightarrow Y$ satisfies the following condition :
there exists $\alpha \in H^{m-n}(X)$ such that $f^!(\alpha)=1$ or equivalently
 $f^*(\mu_Y) \neq 0$ in $H^*(X)$
 then 
$$\gamma (f^{-1}\cF)= \gamma(\cF)$$
In particular for $L$ in $\mathcal L(T^*Y)$ if $f$ is non-characteristic for $L_{1}, L_{2}$ so that $\Lambda_f^{-1}L_{1},  \Lambda_f^{-1}L_{2} \in \mathcal L(T^*X)$, we obtain 
$$\gamma (\Lambda_f^{-1}L)\leq \gamma(L)$$
and under assumption (\ref{Cor-5.5-ii}) 
$$\gamma (\Lambda_f^{-1}L)=\gamma(L)$$
\end{enumerate} 
\end{cor}
\begin{proof} 
For (\ref{Cor-5.5-i})
 the first inequality in Proposition \ref{Prop-4.4} we get $c(1_Y,\cF)\leq c(1_X, f^{-1}\cF)$ hence by duality $c(\mu_Y,\cF)\geq c(\mu_X, f^{-1}\cF)$ and
we get $\gamma (f^{-1}\cF)\leq \gamma(\cF)$. 

For (\ref{Cor-5.5-ii}), we must prove that  the assumption $f^{*}(\mu_Y)\neq 0$ implies$$\gamma (f^{-1}\cF)\geq \gamma(\cF)$$
Indeed, by Proposition \ref{Prop-4.4}, $c(\mu_Y\cF) \leq c(f^*(\mu_Y), f^{-1}\cF)$, but by the triangle inequality (applied to $\cF_2=\cF_3$), since $f^*(\mu_Y)\neq 0$,  by Poincar\'e duality there is a class $\alpha$ such that $f^*(\mu_Y)\cup \alpha =\mu_X$. Then  
$$ c(\mu_Y\cF) \leq c(f^*(\mu_Y), f^{-1}\cF)\leq c(\mu_X, f^{-1}\cF)$$ hence by duality
$$c(1_Y, \cF) \geq c(\mu_X, f^{-1}\cF)$$ and finally $$\gamma (f^{-1}\cF) \geq \gamma(\cF)$$
This concludes our proof  of (\ref{Cor-5.5-ii}).
\end{proof}

\begin{rem} 
When we have a \GFQI , $S: Y \times {\mathbb R}^k$ for $L$ then $\Lambda_f^{-1}\circ L$ has the \GFQI  $S_f$ defined by $S_f(x,\xi)=S(f(x);\xi)$. Then denoting $\tilde f(x,\xi)=(f(x),\xi)$, we have $S_f=S\circ \tilde f$ and $S_f^t=\{(x,\xi) \mid S_f(x,\xi)\leq t\}$ is the preimage by $\tilde f$ of $S^t$. Setting $E_Y=Y\times {\mathbb R}^k, E_X=X \times {\mathbb R}^k$ we get
by duality (\cite{Spanier}, p.296, theorem 17)
$$H_q(S^t,S^{-\infty}) \simeq H^{n-q}(E_Y\setminus S^{-\infty}, E_Y\setminus S^t) =H^{n-q}(E_Y, (-S)^{-t})$$
and similarly 
$$H_q(S_f^t,S_f^{-\infty}) \simeq H^{m-q}(E_Y\setminus S_f^{-\infty}, E_Y\setminus S_f^t) =H^{m-q}(E_Y, (-S_f)^{-t})$$
where $m=\dim(X)+k, n=\dim(Y)+k$. 
Thus the map $$\tilde f_*: H_{q}(E_Y, (-S)^{-t}) \longrightarrow H_{q}(E_Y, (-S_f)^{-t}) $$ induces a map  
$$\tilde f_{!}: H^{m-q}(S^t,S^{-\infty}) \longrightarrow H^{n-q}(S_f^t,S_f^{-\infty})$$ corresponding to $f_! : H^*(X) \longrightarrow H^{*+n-m}(Y)$. The inequality
the follows by the same argument as in the case of sheaves. 
\end{rem} 

Our goal is now to prove an estimate for the reduction going in the opposite direction compared to Proposition \ref{Prop-reduction-inequality}. 
\begin{thm}[Inverse reduction inequality] \label{Thm-Reverse-reduction-inequality}
Let $\widetilde L_1, \widetilde L_2$ be two Lagrangians in $\mathcal L(T^*(X\times Y))$  and $(\widetilde L_1)_x, (\widetilde L_2)_x$ be their  reductions at $x\in X$ assumed to be {\bf embedded}. Let $d=\dim(X)$. 
Then  we have $$c_+(\widetilde L_1,\widetilde L_2) \leq  \sup_{x\in X} c_+((\widetilde L_1)_x, (\widetilde L_2)_x)+ d \sup_{x\in X}\gamma ((\widetilde L_1)_x, (\widetilde L_2)_x)$$ and
$$c_-(\widetilde L_1,\widetilde L_2) \geq  \inf_{x\in X} c_-((\widetilde L_1)_x, (\widetilde L_2)_x)- d\sup_{x\in X} \gamma ((\widetilde L_1)_x, (\widetilde L_2)_x) $$
\end{thm}
  
  We first set 
  \begin{defn}[see \cite{Viterbo-stochastic}, section 5] Let $\cF$ in $\mathcal T_0(X)$. Let  $\Omega$ be a bounded connected  open set in $X$. We denote by $H_{fc}^*(\Omega\times I, \cF)$ where $I\subset {\mathbb R} $ is an interval, the cohomology with support having compact projection in $\Omega$. In other words we consider sections intersecting  $\Omega \times [a,b]$ for any $a,b \in {\mathbb R}$ in a compact set. 
  
   Let $\mu_\Omega$  (resp. $1_\Omega$) be the generator of $$H_{fc}^d(\Omega \times {\mathbb R}, \cF)\simeq H^d(\Omega\times {\mathbb R}  , \partial \Omega \times {\mathbb R} ; 
\cF)\simeq H_{c}^d(\Omega)\simeq H^d(\Omega, \partial \Omega)$$
(resp. 
$H^0(\Omega \times {\mathbb R}, \cF)\simeq H^0(\Omega)$
).
 We set 
$$ c_+(\Omega , \cF)= \inf \left \{ t \mid \mu_\Omega \in \Image (H_{fc}^*(\Omega\times [t, +\infty[,\cF) \right \}$$
$$ c_-(\Omega , \cF)= \inf \left \{ t \mid 1_\Omega \in \Image (H^*(\Omega\times [t, +\infty[,\cF) \right \}$$
Moreover for $x \in N$ we set $$c(x\times [a,b[,\cF)=\lim_{\Omega \ni x} c_+(\Omega \times [a,b[, \cF)=\lim_{\Omega \ni x} c_-(\Omega\times [a,b[, \cF)$$
\end{defn} 
\begin{rem} 
The last equality holds because $\lim_{\Omega \ni x}H^*(\Omega \times {\mathbb R} , \cF)= H^*( {\mathbb R} , \cF_x)$ where $\cF_x$ is the limit of $\cF_{U\times {\mathbb R} }$ for $x \in U$. But $H^*(\{x\}\times  {\mathbb R} , \cF_x)= H^*( \{x\}, k_X)$ is one dimensional. 
\end{rem} 
  \begin{proof}[Proof of Theorem \ref{Thm-Reverse-reduction-inequality}] Note that it is enough to prove the first inequality, the second follows by using $c_-(\widetilde L_1,\widetilde L_2)=-c_+(\widetilde L_2,\widetilde L_1)$. 
  We shall use the inequality from Lemma \ref{Ki-Sh} in Appendix \ref{Appendix1} (this is essentially  in proposition 36 from  \cite{Ki-Sh}), and apply it to $FH_*(\widetilde L_1,\widetilde L_2)$. This  yields $$\beta ((\widetilde L_1)_x, (\widetilde L_2)_x)) \leq \gamma ((\widetilde L_1)_x, (\widetilde L_2)_x))$$ where $\beta (L_1,L_2)$ is the boundary depth of  $FH_*(L_1,L_2)$ (see Appendix \ref{Appendix1} for the definition). Note that this is the step where we need the reductions to be embedded. We shall only use this inequality and the fact that the boundary depth $\beta$ for filtered Floer homology (or for any persistence module) satisfies the following: 
  
  if $x\in FH^*_t(L_1,L_2)$ has image $0$ in $FH^*(L_1,L_2)$ then it has already image $0$ in $FH^*_{t+\beta}(L_1,L_2)$.  

We then set $h^t_U=H_{fc}^*(U \times ]-\infty, t[; \cF)$  and   $h_{\Omega}^{+\infty}= H_{fc}^*(\Omega)$ for any connected open set $\Omega$,  so that we have the following diagram representing the  Mayer-Vietoris exact sequence :
$$\xymatrix {
h_{U\cap V}^t \ar[rr] && h_U^t\oplus h_V^t \ar[dl] \\ &h_{U\cup V}^t \ar[ul]^{+1}&\\
}
 $$
  We denote by $c(U, \cF)$ the smallest $t$ such that the map 
  $$H_{fc}^*(U\times [t, +\infty [, \cF) \longrightarrow H_{fc}^*(U\times ]-\infty, +\infty[, \cF)$$ has $\mu_U$, the generator of $H_{fc}^n(U)$ in its image. 
  We then have the diagram
  
\center{\resizebox{8cm}{!} {\xymatrix {
  &&h^\infty_U\ar[ddrr]^{j_U^\infty}&& \\
  &&h^t_U\ar[u]\ar[dr]^{j_U^t}&&\\
  h_{U\cap V}^\infty \ar[ddrr]_{i_{V}^\infty}\ar[uurr]^{i_{U}^\infty} &\ar[l]  h_{U\cap V}^t \ar[ur]^{i_U^t}\ar[dr]_{i_V^t} & &h^t_{U\cup V} \ar[r]& h^\infty_{U\cup V}\\
&&h_V^t\ar[ur]_{j_V^t}\ar[d]&&\\
&&h_V^\infty\ar[uurr]_{j_V^\infty}&&}
}
}

  Now let us consider elements $a \in h_U^t$ and $b \in h_V^t$ having images in $h_U^\infty, h_V^\infty$ given by  $\mu_U, \mu_V$. Then the images of $a,b$ in $h_{U\cup V}^t$ are not necessarily equal, but since the images of $\mu_U\in h_U^\infty$ and $\mu_V\in  h_V^\infty$ coincide in $h_{U\cup V}^\infty$ - and are equal to $\mu_{U\cup V}$ -  the image of $a$ and $b$ in $h_{U\cup V}^{t+\beta}$ do coincide, and we get a class $c$ in $h_{U\cup V}^{t+\beta}$ having image $\mu_{U\cup V}$ in $h_{U\cup V}^\infty$. 
    \begin{figure}[ht]
\center\begin{overpic}[width=6cm]{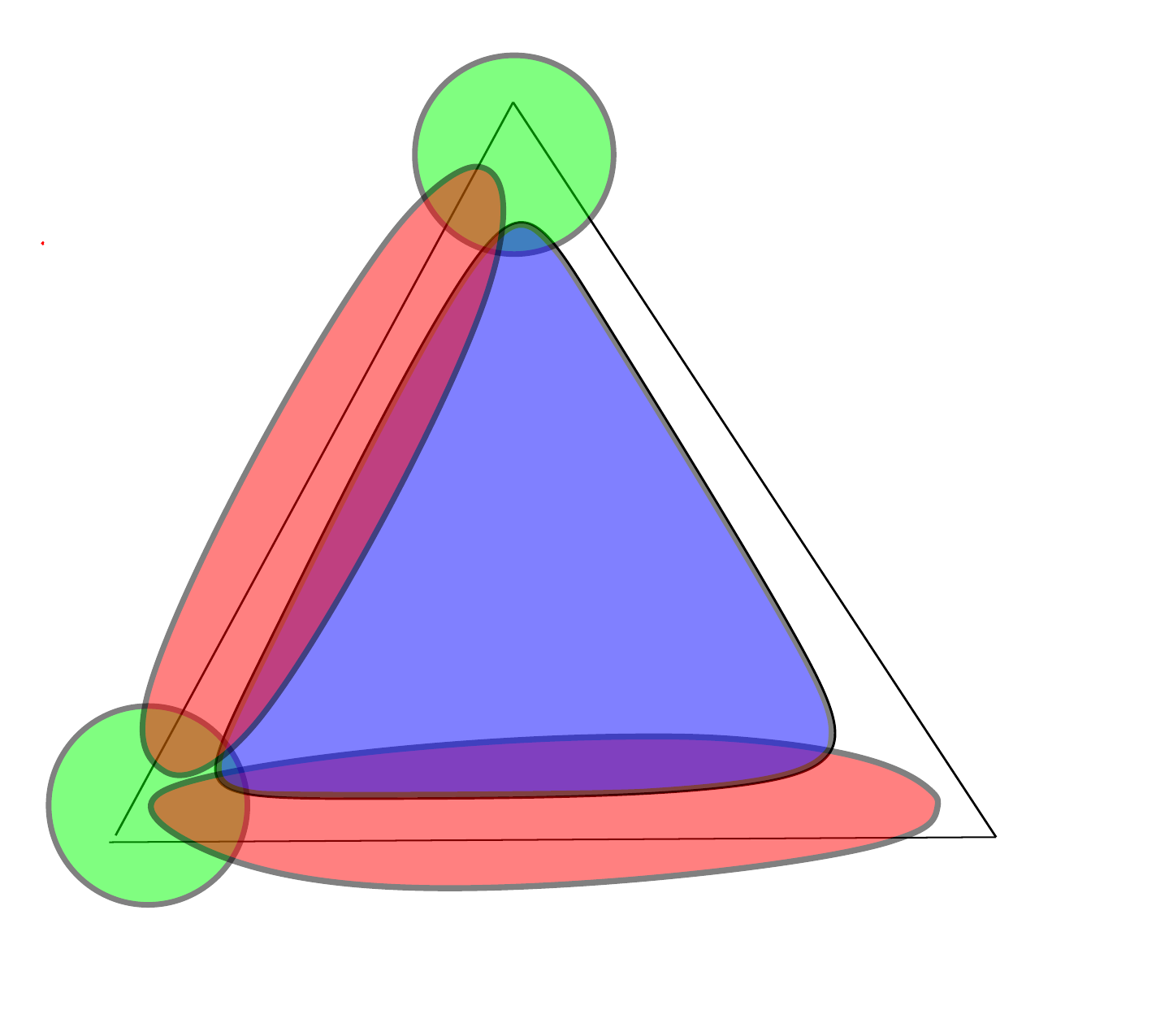}
 \end{overpic}
\caption{The Triangulation}
\label{fig-2}
\end{figure}
  Now consider a covering of $X$ by $d+1$ open sets $U_0, ..., U_d$  such that each $U_j$ is the disjoint union of small open sets $V_{j,k}$ for $1\leq k \leq m_j$, obained from a triangulation as on Figure \ref{fig-2}. We denote by $\mathcal U_j=U_j\times Y, \mathcal V_{j,k}=V_{j,k}\times Y$.  If the triangulation is fine enough, the $V_{j,k}$ are small enough, so that   $h_{\mathcal V_{j,k}}^t \longrightarrow h_{\mathcal V_{j,k}}^\infty$ contains $\mu_{\mathcal V_{j,k}}$ as soon as $t\geq c(x; \cF)+ \eta$ where $c(x; \cF)$ is the smallest $t$ such that $h_x^t \longrightarrow h_x^\infty$ contains $1_x$ and $\eta>0$. Note that indeed $\lim_{\Omega \ni x\times {\mathbb R} }c_\pm(\Omega; \cF)=c(x; \cF)$. Now clearly $c_+(\mathcal U_j; \cF) =\sup_{1\leq k \leq m_j} c_+(\mathcal V_{j,k}; \cF)$ hence  $$c_+(\mathcal U_j; \cF) \leq \sup_{x\in N} c(x ; \cF)+ \eta$$ 
  We just proved that if $\beta$ is an upper bound for the $ \beta (x; \cF)$ we have $$c_+(\mathcal U_0\cup ... \cup \mathcal U_{j}\cup \mathcal U_{j+1}; \cF) \leq \max\left\{c_+(\mathcal U_0\cup ... \cup \mathcal U_{j}), c_+(\mathcal U_{j+1}; \cF)\right \}+\eta + \beta$$
  It follows by induction that
  $$c_+(\mathcal U_0\cup ... \cup \mathcal U_{j}\cup \mathcal U_{d}; \cF) \leq c_+(\mathcal U_0; \cF) + d \beta = d \eta+ d\beta \leq d\eta + (d+1) \sup_{x\in X}\gamma (x; \cF)$$
  Since this holds for any positive $\eta$, this concludes the proof. 
  \end{proof} 
  
  Note the following Lemma obtained at the beginning of our proof. 
  \begin{prop} 
  Let $U, V$ be open sets with boundary, and set $\gamma(U,V;\cF)=\max \left \{ \gamma(U,\cF), \gamma(V;\cF)\right\}$,  $\beta(U,V;\cF)=\max \left \{ \beta(U,\cF), \beta(V;\cF)\right\}$. Then 
 \begin{gather*}  c_+(U\cup V; \cF) \leq \max \left \{ c_+(U,\cF), c_+(V;\cF)\right \}+ \beta(U,V;\cF) \leq \\ \max \left \{ c_+(U,\cF), c_+(V;\cF)\right \}+ \gamma(U,V;\cF)
 \end{gather*} 
  \end{prop} 
    \begin{rems}\label{Rems-5.12}
    \leavevmode
\begin{enumerate} 
\item  In the framework of \GFQI, the proof becomes slightly more intuitive.   Let us first explain the case $X=S^1$. Set  $S_x(y,\xi)=S(x,y;\xi)$ where $S$ is a \GFQI for $L$, then there is a cycle $C_x \subset S_x^{c_x}$ representing the generator of $H_*(S_x^\infty, S_x^{-\infty})$ for $c_x\leq c_+(S_x)+ \varepsilon $. We may divide $S^1= [0,1]/\simeq $ in $N$ subintervals $[ \frac{k-1}{N}, \frac{k}{N}]$ so that $C_{ \frac{k-1}{N}}$ is contained in    $S_x^{c_{k-1}}$ for all $x\in [ \frac{k-1}{N}, \frac{k}{N}]$. Of course at the boundary points $ \frac{k}{N}$ , we have two cycles:  the one coming from $[ \frac{k-1}{N}, \frac{k}{N}]$ that we denote by $C_k^-$ and the one coming from $[ \frac{k}{N}, \frac{k+1}{N}]$ denoted by $C_k^+$. Now $[C_k^-]=[C_k^+]$ in $H_*(S_{ \frac{k}{N}}^\infty, S_{ \frac{k}{N}}^{-\infty})$ but this implies by definition of the boundary depth that $C_k^--C_k^+ =\partial \Gamma_k$ with $\Gamma_k \in S^{c_k+ \beta (S_{\frac{k}{N}})}$. Then gluing toghether the $[ \frac{k-1}{N}, \frac{k}{N}] \times C_k$ and the $\Gamma_k$ for $k$ in $[0,N]$, we get a cycle $\widetilde C$ that represents the generator of $H_*(S^\infty, S^{-\infty})$ and contained in $S^c$ for $c \leq \sup_x c_+(S_x)+\sup_x \beta (S_x) \leq 2 \sup_x(S_x)=2 \sup_x c_+(L_x)$.

  As a result $c_+(L) \leq 2 \sup_x c_+(L_x)$ hence $\gamma (L) \leq \sup_x c_+(L_x)$. 
  The higher dimensional case is analogous (see \cite{SHT}, lemma 6.1) by triangulating $X$. We start by taking constant cycles at the interior of a triangle (on the picture on the blue region), then in the red region we have and then glue at the common side with a chain $\Gamma$ such that $\partial \Gamma = C_+ \cup -C_-$ and $\Gamma \subset S^{2c}$ where $c=\gamma(S_x) \geq  \beta (S_x)+ \varepsilon $. At a vertex $x_0$, we get $C_1,..., C_k$ and we just constructed $\Gamma_{i,j}$ such that $\partial \Gamma_{i,j}=C_i - C_j$. Thus summing over all $i,j$  we get that $\partial (\sum_{i<,j} \Gamma_{i,j} )=\sum_{i<j} C_i- C_j=0$. So $\sum_{i<,j} \Gamma_{i,j}$ is a cycle
  and is a cocycle in $H_*(S^\infty_{x_0}, S^{-\infty}_{x_0})$ so can be written as $\partial T_{x_0}$  in $S_{x_0}^{3c}$. We can then glue the $C_i$, the $\Gamma_{i,j}$ and $T_{x_0}$ to a cycle in $S_{x_0}^{3c}$.  

\item As pointed out by B. Chantraine, we do not really need that $(L_1)_x, (L_2)_x$ be embedded. It is enough that $FH^*((L_1)_x, (L_2)_x)$ has a unit satisfying the Conditions (\ref{Cond-2}), (\ref{Cond-4}) of Appendix \ref{Appendix1}.  
\item We do not necessarily have 
$$\gamma (L_1,L_2) \leq C_d \sup_x \gamma ((L_1)_x, (L_2)_x))$$
Indeed, if $L_2=0_{X\times Y}$ and $L_1= \gr (df)$ for $f \in C^\infty(X\times Y, {\mathbb R})$ the inequality we proved  translates to 
$$\sup_{X\times Y} f\leq C_d \sup_{x\in X} \sup_Y(f_x)$$ which obviously holds, while the above inequality would translate to 
$$ \osc_{X\times Y} (f) \leq C_d \sup_{x\in X} \osc_Y(f_x)$$
which does not hold in general : if $f$ is constant on $Y$ the RHS vanishes, while the LHS can well be non-zero. 
However we can often recover  an inequality on the spectral norms as above for special cases of $L_1, L_2$. 
 
Note that it is always more convenient to deal with the sheaves (or the  \GFQI) than with the Lagrangians. 
\item \label{Rems-5.12-d}
Note however that if we know that $c_-((\widetilde L_1)_x,(\widetilde L_2)_x)\leq 0 \leq c_+((\widetilde L_1)_x,(\widetilde L_2)_x)$ then we get 
$$\gamma (L_1,L_2) \leq C_d \sup_x \gamma ((L_1)_x, (L_2)_x))$$ since then $c_+((\widetilde L_1)_x,(\widetilde L_2)_x) \leq \gamma((L_1)_x,(L_2)_x) $ and $-c_-((\widetilde L_1)_x,(\widetilde L_2)_x) \leq \gamma((L_1)_x,(L_2)_x)$. 

\end{enumerate} 
\end{rems} 

\section{A remarkable property of the spectral distance}\label{Section-6}
In this section $N$ will be a closed $n$-dimensional manifold. 
\begin{defn} \label{Def-2.1} We say that $N$ satisfies Condition ($\star$) if
there exists  a closed manifold $V$ and a map $\Phi: V \longrightarrow \Diff_0(N)$ such that the map $\Phi_{x_0}=ev_{x_0} \circ \Phi: V \longrightarrow N$ satisfies $\Phi_{x_0}^*(\mu_N)\neq 0$ in $H^*(V)$. 
\end{defn} 

Until the end of  this section and the next one we assume $N$ satisfies Condition ($\star$). Note that any Lie group $G$, obviously satisfies condition ($\star$), by taking $V=G$ and $\Phi(g)(x)=g\cdot x$, so that clearly $\Phi_{e}=\Id_G$. 
We first need the following folk construction, going back at least to Arnold (\cite{Arnold-cobordism}).
\begin{lem}\label{Lemma-7.5} %
Let $\widetilde L\in \mathcal L (T^*N)$ and $H\in C^\infty(X\times N\times [0,1], {\mathbb R} )$ so that for each $x\in X$ we can consider the Hamiltonian $H_x(t,z)$ yielding a flow $\varphi_x^t$ on $T^*N$. Then there is a well defined $\widetilde\Lambda\in \mathcal L(T^*(X\times N))$ such that the reduction of $\widetilde \Lambda \in \mathcal L (T^*(X\times N))$ by $T_x^*X\times T^*N$ is   $\varphi_x^1(\widetilde L)$.  
\end{lem}  
\begin{proof} See e.g. lemma 5.23 from \cite{Viterbo-book}. 
Note that by assumption  $\widetilde L_x=\varphi_x^1(\widetilde L)$ has a well defined primitive for $\lambda=pdq$ that we denote by $f_x(q,p)$. 
 Then $\Lambda =\{(x, y(x,q,p), \varphi_x^1(q,p)) \mid y=f_x(q,p)\}$  and $F(x,q,p)=f_x(q,p)$ must be a primitive for $ydx+ pdq= dF(x,q,p)$. 
 Alternatively  $H(t,x,q,p)$ defines a flow $\Phi_H^t$ on $T^*(X\times N)$ and $\Phi_H^1(0_X\times L)=\Lambda$. 
\end{proof} 

We are first going to prove the following remarkable property

\begin{thm} \label{Thm-6.3}
Let $L$ be an exact Lagrangian in $T^*N$ where $N$ satisfies the above Condition ($\star$). 
There exists $\varphi\in V\subset \Diff_0(N)$ such that its lift $\tau_\varphi$ to $T^*N$ satisfies $\gamma (L, \tau_\varphi (L)) > \frac{\gamma(L)}{n+2}$. 
\end{thm} 

\begin{proof} 
Remember that the lift $\tau_\varphi$ of $\varphi$ is given by $(x,p) \mapsto (\varphi(x), p\circ d\varphi(x))$, and if $\varphi$ is the time one of the isotopy generated by the vector field $X(t,x)$, then $\tau_\varphi$ is the Hamiltonian map associated to $H(t,x,p)=\langle p, X(t,x)\rangle$. As a result for $\varphi \in V$ and $\widetilde L\in \mathcal L (T^*N)$ then $\tau_\varphi(\widetilde L)$ is well defined as an element in $\mathcal L (T^*N)$.

Let us denote by $\varphi$ an element in $V$, and set $\varphi(x)=\Phi(\varphi)(x)$. We consider the map $\Psi : V\times N \longrightarrow N$ given by $\Psi(\varphi, x)=\varphi(x)$.%
We set $\widetilde  L_\varphi= \tau_\varphi (\widetilde L)$. We want to find $\varphi$ so that 
$c_+(\widetilde L_\varphi, \widetilde L)$ is bounded from below.
Let $\cF_L$ be the sheaf associated to $\widetilde L$ so that $(\varphi)_*\cF_L$ is associated to $\tau_\varphi(\widetilde L)$, and let us  compute $c_+(\widetilde L_\varphi, \widetilde L)$. Let $\widetilde \Lambda$ be the Lagrangian in $T^*(V\times N)$ given by Lemma \ref{Lemma-7.5}. Its  reduction at $\varphi=\varphi_0$ is given by $\tau_{\varphi_0}( \widetilde L)$.
Now, according to the inverse reduction inequality (Theorem \ref{Thm-Reverse-reduction-inequality}), since $L_\varphi$ (resp.  $L$)  is embedded and coincides with  the reduction of $\widetilde \Lambda$ (resp. $0_V\times L$) at $\{\varphi\}\times N$
$$c_+(\widetilde \Lambda , 0_{V}\times \widetilde L)\leq \sup_{\varphi \in V} c_+(\widetilde L_\varphi , \widetilde L) + n \sup_{\varphi \in V} \gamma ( L_\varphi, L)$$
Note that 
 since  $$c_+( \widetilde L,\tau_\varphi(\widetilde L))= c_+(\tau_{\varphi^{-1}}( \widetilde L),\widetilde L)$$
 we have
$$c_-(\widetilde L_\varphi, \widetilde L)=-c_+( \widetilde L,\widetilde L_\varphi)= -c_+(\widetilde L_{\varphi^{-1}}, \widetilde L)$$
and therefore
$$ \inf_{\varphi \in V} c_-(\widetilde L_\varphi , \widetilde L)= \inf_{\varphi \in V} -c_+( \widetilde L,\widetilde L_\varphi)= -\sup_{\varphi \in V}c_+(\widetilde L_{\varphi^{-1}}, \widetilde L)
$$
As a result $$\sup_{\varphi \in V} c_+(\widetilde L_\varphi , \widetilde L)-\inf_{\varphi \in V} c_-(\widetilde L_\varphi , \widetilde L) \leq \sup_{\varphi \in V} c_+(\widetilde L_\varphi , \widetilde L) +\sup_{\varphi \in V} c_+(\widetilde L_{\varphi^{-1}} , \widetilde L)
$$

Using again the inverse reduction inequality %
\begin{gather*} c_-(\widetilde \Lambda , 0_{V }\times \widetilde L)\geq \inf_{\varphi \in V} c_-(\widetilde L_\varphi , \widetilde L) - n \sup_{\varphi \in V} \gamma (\widetilde L_\varphi, L)=-\sup_{\varphi\in V}c_+(\widetilde L_{\varphi^{-1}}, \widetilde L)- n \sup_{\varphi \in V} \gamma ( L_\varphi, L)
\end{gather*} 
so that 
$$-c_-(\widetilde \Lambda , 0_{V }\times \widetilde L)\leq \sup_{\varphi\in V}c_+(\widetilde L_{\varphi^{-1}}, \widetilde L)+ n \sup_{\varphi \in V} \gamma (\widetilde L_\varphi, L)
$$
hence adding the inequalities for $c_{+}$ and $- c_{-}$ we get
 \begin{equation}\label{Eq-7.1} \gamma(\widetilde \Lambda , 0_{V }\times \widetilde L) \leq \sup_{\varphi\in V}c_+(\widetilde L_{\varphi^{-1}}, \widetilde L)+\sup_{\varphi\in V}c_+(\widetilde L_{\varphi}, \widetilde L)+n \sup_{\varphi \in V} \gamma ( L_\varphi, L)
\end{equation} 
Applying  the triangle inequality (Proposition \ref{Prop-triangle}), we have
$$c_+(\widetilde L_{\varphi}, \widetilde L)\geq c_+(\widetilde L_\varphi, 0_{N}) - c_-(\widetilde L, 0_{N})$$ and since 
$$c_+(\widetilde L_\varphi, 0_{N})=c_+(\tau_\varphi(\widetilde L), 0_{N})=c_+(\tau_\varphi(\widetilde L), \tau_\varphi(0_{N}))= c_+(\widetilde L, 0_{N})$$
we get 
$$c_+(\widetilde L_{\varphi}, \widetilde L)\geq c_+(\widetilde L, 0_{N}) - c_-(\widetilde L, 0_{N}) \geq 0$$
and similarly $c_-(\widetilde L_{\varphi}, \widetilde L)\leq 0$, so that 
$$c_+(\widetilde L_{\varphi}, \widetilde L)\leq    \gamma (\widetilde L_{\varphi}, \widetilde L)$$
and $$c_+(\widetilde L_{\varphi^{-1}}, \widetilde L)\leq \gamma (\widetilde L_{\varphi^{-1}}, \widetilde L)= \gamma (\widetilde L_{\varphi}, \widetilde L) $$
hence (\ref{Eq-7.1}) becomes
\begin{equation} 
\gamma(\Lambda , 0_{V }\times L) \leq (2+n) \sup_{\varphi \in V} \gamma ( L_\varphi, L)
\end{equation} 

On the other hand the ordinary reduction inequality at $x=x_0\in N$ yields
\begin{equation} 
\gamma(\Lambda , 0_{V }\times  L) \geq \sup_{x_0\in N} \gamma ((\cF_\Lambda)_{x_0}, (\cF_{0_V\times L})_{x_0})
\end{equation} 
Let us denote by $\Phi_{x_0}$ the map $\varphi \mapsto \varphi(x_0)$. Then $\Phi_{x_0}$ is a map from $V$  to $N$
but $(\cF_\Lambda)_{x_0}\simeq \Phi_{x_0}^{-1}\cF_{L}$ and $(\cF_{0_V\times L})_{x_0}= k_{V\times [0, +\infty [}\otimes(\cF_L)_{x_0}$. 
As a result 
$$\gamma ((\cF_\Lambda)_{x_0}, (\cF_{0_V\times L})_{x_0})=
\gamma ((\cF_{\Lambda})_{x_0}, k_{V\times [0, +\infty [}\otimes (\cF_L)_{x_0}) 
\geq \gamma ((\cF_\Lambda)_{x_0}))-\gamma ((\cF_L)_{x_0})$$ 
But since $1=\mu$ in $H^*(\{x_0\})$ we have 
 $$\gamma ((\cF_L)_{x_0})=c_+((\cF_L)_{x_0})-c_-((\cF_L)_{x_0})=0$$
hence
 $$\gamma (\Lambda, 0_{V}\times L)=\gamma ((\cF_\Lambda)_{x_0}, (\cF_{0_V\times L})_{x_0})=\gamma ((\cF_\Lambda)_{x_0}, (\cF_L)_{x_0}) \geq \gamma ((\cF_\Lambda)_{x_0})$$
On the other hand, since $\Phi_{x_{0}}^{*}(\mu_{N})\neq 0$ by Condition ($\star$) we may apply
 Corollary \ref{Cor-5.5} (\ref{Cor-5.5-ii}),  so we get
$$\gamma ((\cF_\Lambda)_{x_0})=\gamma (\Phi_{x_0}^{-1}\cF_L)=\gamma(\cF_L)$$
 We may then conclude that 
 $$\gamma(L) \leq  \gamma(\Lambda , 0_{V }\times L) \leq (n+2) \sup_{\varphi \in V} \gamma (L_\varphi, L)$$
 This concludes the proof.
 \end{proof} 
     \begin{figure}[ht]
\center\begin{overpic}[width=8cm]{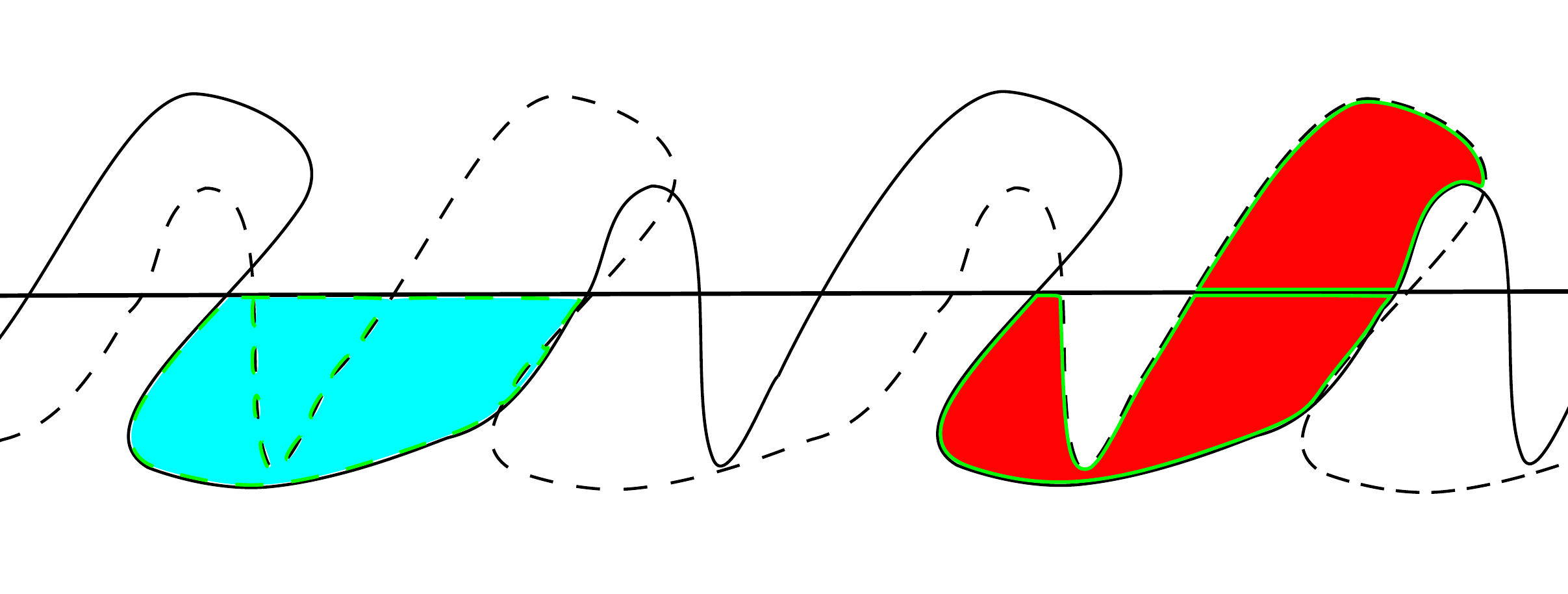}
 \end{overpic}
\caption{Illustration of Theorem \ref{Thm-6.3} for $N=S^1$: the red area, $\gamma(\varphi(L),L)$  is greater than a third of the blue area, $\gamma(L)$. }
\label{fig-3}
\end{figure}

Using the previous result we may prove 
\begin{prop} \label{Prop-6.4}
Let $P=G/H$ be a homogeneous space of dimension $d$, with an action of a compact Lie group $G$ of dimensions $n$. Then there exists $g\in G$ such that 
$$\gamma(L) \leq (d+1)(n+2) \sup_{g\in G} \gamma (L, \tau_g(L))$$ where $\tau_g: T^*P \longrightarrow T^*P$ is induced by multiplication by $g$ on $T^*P$, i.e. $\tau_g(x,p)=(g\cdot x, p\circ dg^{-1})$. 
\end{prop}  \begin{proof} 
Set $P=G/H$ and let $m : G\times P \longrightarrow P$ be the map defining the action. For $g\in G, y\in P$ set $m_{g}(y)=m(g,y), m_{y}(g)=m(g,y)$ (this should cause no confusion). Set $L' = \Lambda_{m}^{-1}(L)$, where $L \subset DT^{*}P$. Then $L' \subset T^{*}(G\times P)$ is given by 
 $$L' = \left\{(g,p_{g}, y,p_{y}) \mid m(g,y)=x, p_{x}dm(g,y)=p_{g}dg+p_{y}dy, (x,p_{x})\in L\right \}$$
The reduction of $L'$ at $g=g_{0}$ is
 $$ L'_{g_{0}}= \left\{(y,p_{y}) \mid \exists (x,p_x )\in L, m_{g_{0}}(y)=x, p_{x}dm_{g_{0}}(y)=p_{y}\right \}= m_{g_{0}}^{*}(L)$$
 and at $y=y_{0}$ it is  
 $$L'_{y_{0}}= \left\{(g,p_{g}, y,p_{y}) \mid m_{y_{0}}(g)=x, p_{x}dm_{y_{0}}(g)=p_{g}dg,  (x,p_{x})\in L\right \}=m_{y_{0}}^{*}(L)$$
 Since $dm_{y_{0}}$ is onto,  $m_{y_{0}}^{*}(L)$ is embedded (and of course, so is $m_{g_{0}}^{*}(L)$). 
 We then have by the inverse reduction inequality  $$\gamma(L') \leq (d+1) \sup_{y_0}\gamma(m_{y_{0}}^{*}(L)) 
 $$   Since $m_{y_0}^*(L)$ is in $T^*G$ we have by Theorem \ref{Thm-6.3}
$$ \gamma(m_{y_{0}}^{*}(L))  \leq (n+2) \sup_{g\in G} \gamma (m_{y_0}^*(L), \tau_gm_{y_0}^*(L))=(n+2) \sup_{g\in G} \gamma (m_{y_0}^*(L), m_{y_0}^*(\tau_g(L)))$$  
Now applying the inequality from Corollary \ref{Cor-5.5} we have
 $$\gamma (m_{y_0}^{*}(L_1), m_{y_0}^{*}(L_2))\leq \gamma (L_1, L_2)$$  hence 
 $$\gamma(L') \leq (n+2)  \sup_{g\in G} \gamma (m_{y_0}^*(L), m_{y_0}^*(\tau_g(L)))\leq  (n+2)\sup_{g\in G} \gamma (L, \tau_g(L))$$
 The same inequality implies that $\gamma(m_e^*(L')) \leq \gamma (L')$ so that finally, using $m_e^*(L')=L$ we may conclude
 $$\gamma(L) \leq  (d+1)(n+2)\sup_{g\in G} \gamma (L, \tau_g(L))$$

\end{proof} 
 Finally we state the following 
  
\begin{conjecture}
Let $(M, \omega)$ be a symplectic manifold. Let $L$ be an exact Lagrangian. There exists a constant $c$ depending on the pair $(M, L)$  such that for any $L_1$ Hamiltonianly isotopic to $L$, such that $\gamma(L,L_1)$ is small enough, there exists a Hamiltonian isotopy $\varphi^t$ preserving $L$ such that $$\gamma (L_1, \varphi^1(L_1))\geq c \gamma (L,L_1)$$
\end{conjecture}
\begin{rem} 
 Note that $\varphi^t$ is an isometry for $\gamma$. For $\varphi^{t}$ preserving $0_{N}$, 
we can consider the triangle $0_N, L, \varphi^t(L)$. It is clearly isosceles, since $\gamma(L,0_N)=\gamma (\varphi^t(L),0_N)$ because $\varphi^t$ preserves $0_N$. Now  $$\gamma(L, \varphi^1(L))\geq c \gamma (0_N,L)$$ so we may find $t_0\in [0,1]$ such that $\gamma (L, \varphi^{t_0}(L))=c\gamma (L, 0_N)$.  In particular if $c=1$ then the triangle $(0_N,L, \varphi^{t_0}(L))$ is equilateral. Existence of equilateral triangles in metric spaces is an interesting question that was explored in (see \cite{Furstenberg-Katznelson-Weiss, Iosevich-equilateral}). If $c<1$ as in our case, we still get ``wide isosceles triangles''.  
\end{rem}

 \section{Spectral boundedness of geometrically bounded Lagrangians}
 Let $N$ be a closed manifold endowed with a Riemannian metric $g$. We denote by $D_rT^*N$ the radius $r$ cotangent disc bundle, and set $DT^*N=D_1T^*N$. 
 
 We shall deal in this section with the following conjecture from \cite{SHT}

 \begin{conjecture}[Bounded Lagrangians are spectrally bounded]\label{Conjecture-Viterbo}
 Let $N$ be a closed manifold. There exists a constant $C_N$ such that for any 
 $L \in {\mathfrak L}(T^*N)$ such that $L \subset DT^*N$, we have $$\gamma (L) \leq C_N$$
 \end{conjecture}

 We claim
 \begin{prop} \label{Prop-Viterbo-conjecture}
 Let  $N$ satisfy  Condition  ($\star$), and $DT^*N$ be the unit disc bundle of $T^*N$ for some metric $g$.  Then there exists a constant $C_N$ such that for any $L \in \LL (T^*N)$ contained in $DT^*N$ we have $$\gamma (L) \leq C_{N,g}$$
 \end{prop} 

First we need the easy  
 \begin{lem} \label{Lemma-7.3}
 Let $\varphi^t$ be the Hamiltonian flow of $H(t,q,p)$ in $T^*N$ and $L\in {\LL} (T^*N)$. Then we have the inequality
 $$\gamma (\varphi^1(L), L) \leq 2\Vert H_{\mid L} \Vert _{C^0}$$
 where $$\Vert H_{ \mid L} \Vert _{C^0}=\sup\left \{ \vert H(t,q,p)\vert  \mid t\in [0,1], (q,p)\in L \right \}$$
 \end{lem} 
 \begin{proof} Indeed $c_+(\varphi^t_H(\widetilde L), \widetilde L)$ is the action of a trajectory of $\varphi_H^t$ from $L$ to $L$, i.e.  $$c_+(\varphi_H^t(\widetilde L), \widetilde L)=A_t(\gamma_t)=\int_0^t [p_t(s)\dot q_t(s) - H(s,q_t(s),p_t(s)) ]ds $$
 where $\gamma_t(s)=(q(s),p(s))=\varphi^s_H(q_t(0),p_t(0))$ and $(q_t(0),p_t(0))$ and $(q_t(1),p_t(1)$ belong to $ L$. 
 Then it is classical that $t \mapsto A_t(\gamma_t)$ is continuous and piecewise $C^1$, and for almost all $t\in [0,1]$, 
 $$ \frac{d}{dt} A_t(\gamma_t)=-H(t,q_t(t),p_t(t))$$ so that 
 $$\left \vert \frac{d}{dt} A_t(\gamma_t) \right \vert \leq \Vert H_L \Vert_{C^0}$$ for almost all $t$, hence 
 $$ c_+(\varphi^t_H(\widetilde L), \widetilde L)- c_+(\varphi^0_H(\widetilde L), \widetilde L)=A_t(q_t(t),p_t(t))-A_0(q_0(0),p_0(0))  \leq \vert t \vert \Vert H_L \Vert_{C^0}
 $$
 We have a similar inequality for $c_-$ and the estimate for $\gamma$ immediately follows. 
 \end{proof} 
   \begin{proof}[Proof of Proposition \ref{Prop-Viterbo-conjecture}]
 According to Proposition \ref{Thm-6.3} since $N$ satisfies  condition ($\star$),  there exists $\varphi \in V$ such that  $$\gamma (\tau_\varphi L, L) \geq \frac{\gamma(L)}{2+n}$$ where $\tau_\varphi$ is the lift of some diffeomorphism of the base isotopic to the identity. 

 But $\tau_\varphi (L)$ is obtained by applying to $L$ the flow of a Hamiltonian  $H(t,q,p)=\langle p , X_t(x) \rangle$ in $DT^*N$. 
  Provided Condition ($\star$) holds, we proved in Proposition \ref{Thm-6.3} that there exists $\varphi \in V$ such that $C\gamma(\tau_\varphi L,L) \geq  \gamma(L)$. 
  But $\varphi$ is the time one flow of the vector field $X_t$, so $\tau_\varphi$ is the flow of $H(t,x,p)=\langle p, X_t(x,p)\rangle$. If $C$ bounds the norm of $H$ on $D_1T^*N$,  noting that such a bound exists by compactness of $V$ and only depends on the set $V$, we have 
  $$2C \geq \gamma ( \tau_\varphi L, L) \geq \frac{1}{2+n}\gamma (L)$$ the first inequality is a consequence of  Lemma \ref{Lemma-7.3}, the second of Proposition \ref{Thm-6.3}. 
  
This concludes the proof of our Proposition.  
 \end{proof}

We also have
 \begin{prop} \label{Prop-7.6}
 Let $G/H$ be a homogeneous space associated to compact Lie group $G$. Then $G/H$  satisfies Conjecture \ref{Conjecture-Viterbo}. 
 \end{prop} 
 \begin{proof} 
 Using Proposition \ref{Prop-6.4} the proof is the same as that of Proposition \ref{Prop-Viterbo-conjecture}. 
 \end{proof} 

 \begin{rem} 

A similar result (the case $N$ is a homogeneous space) was proved independently by Guillermou and Vichery (\cite{Guillermou-Vichery}). 
 \end{rem} 
 The result easily extends to more general manifolds by
 \begin{prop} \label{Prop-7.5}
Let $f : P \longrightarrow N$ be a map between closed manifolds such that $f^{*}(\mu_{N})\neq 0$ and assume there is a constant $C_{P}$ such that for $L\in \mathfrak L(T^{*}N)$ contained in $DT^{*}P$ we have $\gamma(L)\leq C_{P}$ (i.e. $P$ satisfies Conjecture \ref{Conjecture-Viterbo}). Then for all $L$ in $DT^*N$, we have
$$\gamma (L)\leq C_P$$ (so $N$ satisfies Conjecture \ref{Conjecture-Viterbo})
 \end{prop} 
 \begin{proof} 
  Indeed, 
 if $L$ is exact Lagrangian in $T^*N$, we can lift $L$ to an exact Lagrangian $L'$ in $T^*P$ by setting
$$L'=\{(y,p_y) \mid p_y=p_x\circ d\pi, (\pi(y), p_x)\in L\}=\Lambda_{\pi}^{-1} L$$  
Note that we may assume by compactness that we have a Riemannian metric on $P$ such that the metric on $N$ is dominated by the metric on $P$ (i.e. $\pi$ is a contraction), then 
if $L \subset DT^*N$ we have $L' \subset DT^*P$. As a result we have $\gamma (L') \leq C_P$ but by  Corollary \ref{Cor-5.5} (\ref{Cor-5.5-ii})  we
know that $\gamma (L)= \gamma(L')$, so $\gamma(L)\leq C_{P}$. 
\end{proof}

To summarize our findings, we set
  \begin{defn} \label{Def-6.5}
  We denote by $\mathcal V$ the class of closed manifolds 
  \begin{enumerate} 
  \item \label{Def-6.5-i}Containing all manifolds satisfying Condition $(\star)$, and in particular all compact Lie groups. 
  \item  \label{Def-6.5-ii} Containing all homogeneous spaces  of compact Lie groups. 
  \item \label{Def-6.5-iii} If $P\in \mathcal V$ and $f: P \longrightarrow N$ is a map such that $f^{*}\mu_{N}\neq 0$ in $H^{*}(P)$ then  $N \in \mathcal V$. 
  \end{enumerate} 
  \end{defn} 
\begin{rems}
\begin{enumerate} 
\item If $f: P \longrightarrow N$ is a non-constant holomorphic map between K\"ahler manifolds, the condition of the Corollary holds, since according to Blanchard and Deligne (see \cite{Blanchard, Deligne-degeneration} the Leray spectral sequence degenerates at $E_{2}$. This implies that $E_{2}^{n,0}=H^{n}(N, R^{0}f_{*}( {\mathbb R}_{P}))$ survives to $H^{n}(P, {\mathbb R})$, i.e. that $f^{*}(\mu_{N})\neq 0$. 
 \item We do not really need the full power of Condition ($\star$), it is enough to have some  map $f: P \longrightarrow N$ and  a map $\Phi: V \longrightarrow \Diff (P)$ such that the map $\Psi_{y_0}: \varphi \mapsto f(\varphi(y_0))$ satisfies $\Psi_{y_0}^*\mu_N \neq 0$. 
\item Note the following strange phenomenon : according to Shelukhin (see  \cite{shelukhin-sc-viterbo}) the condition for a manifold to  be ``string-point invertible'' implies that Conjecture \ref{Conjecture-Viterbo} holds in $T^{*}N$. Now the property of being ``string point invertible'' is stable by taking products. This also holds for the class $\mathcal V$ since 
\begin{enumerate} 
\item The product of two manifolds satisfying Condition $(\star)$ satisfies Condition  $(\star)$
\item Products of homogeneous spaces are homogeneous spaces
\item If $P_1,P_2 \in \mathcal V$ and $f_{j}: P_{j } \longrightarrow N_{j}$ are maps such that $f_{j}^{*}(\mu_{N_{j}})\neq 0$ then  $f_{1}\times f_{2}: P_{1}\times P_{2} \longrightarrow N_{1}\times N_{2}$ satisfies $(f_{1}\times f_{2})^{*}(\mu_{N_{1}\times N_{2}})= f_{1}^{*}(\mu_{N_{1}})\otimes  f_{2}^{*}(\mu_{N_{2}})\neq 0$
\end{enumerate} 
Moreover  if $N_{1}\times N_{2}$ satisfies Conjecture \ref{Conjecture-Viterbo},  then according to Corollary \ref{Cor-5.5}, (\ref{Cor-5.5-ii}) this is also the case for each of the factors $N_{1}$ and $N_{2}$.
\item According to work by Pedro Ontaneda (\cite{Ontaneda, Ontaneda-2}), for any manifold $N$ there is always a non-zero degree map from a negatively curved manifold (and even negative curvature arbitrarily close to $-1$) to $N$. As a result, if we knew that Conjecture \ref{Conjecture-Viterbo} holds for negatively curved closed manifolds it would hold for all manifolds. Strangely enough, as far as we know the conjecture is not known for ANY negatively curved manifold.  So this is where one should look for counterexamples !  
 \end{enumerate} 
\end{rems} 
 
 We have then proved
 \begin{thm}\label{Thm-viterbo-conjecture}
 For any manifold $N$ in $\mathcal V$, there is a constant $C_N$ such that for any element $L\in \mathcal L(T^*N)$ contained in $DT^*N$, we have
 $$\gamma (L) \leq C_N$$
 \end{thm} 
\begin{proof} 
We proved the estimate when $N$ is given by \ref{Def-6.5} (\ref{Def-6.5-i}). Case \ref{Def-6.5-iii} is a consequence of Proposition \ref{Prop-7.6}.
In case (\ref{Def-6.5-iii})  we 
apply Proposition \ref{Prop-7.5}.
\end{proof} 
We remind the reader of the following generalization of Conjecture \ref{Conjecture-Viterbo} from  \cite{Viterbo-gammas} (we refer to it for all definitions), where $\widehat{\mathfrak L}(T^*N)$ is the completion of 
$({\mathfrak L}(T^*N),\gamma)$) (the Humili\`ere completion) and the support of $L \in \widehat{\mathfrak L}(T^*N)$ has been defined in \cite{Viterbo-gammas} : 
\begin{conjecture}\label{Conjecture-6.8}
 There exists a constant $C_N$ such that for any 
 $L \in \widehat{\mathfrak L}(T^*N)$ with $\gamma-\supp(L) \subset DT^*N$, we have $$\gamma (L) \leq C_N$$
 \end{conjecture}
 We may thus prove
 \begin{prop} \label{Prop-7.11}
 Let $N \in \mathcal V$. Then there exists a constant $C_N$ such that  for any $L \in \widehat{\mathfrak L}(T^*N)$ such that $\gamma-\supp(L) \subset DT^*N$, we have $$\gamma (L,0_N) \leq C_N$$
 \end{prop} 
\begin{proof} 
Let us first remark that according to Lemma 8.9 of \cite{Viterbo-gammas} if $H$ is a Hamiltonian equal to $a$ in a neighbourhood of $\gamma-\supp(\widetilde L)$ then $c_\pm(\varphi_H(\widetilde  L), \widetilde L)=a$. As a consequence  if $\gamma-\supp(L) \subset DT^*N$ and $ -C\leq H \leq C$ on $DT^*N$ we have $c_+ (\tau_{\varphi_H} (\widetilde L), \widetilde L) \leq C$ and $c_- (\tau_{\varphi_H} (\widetilde L), \widetilde L) \geq -C$, so $\gamma (\tau_{\varphi_H} (L), L) \leq 2C$. 

 Now according to Theorem \ref{Thm-6.3} we have $$\gamma(L) \leq C'\sup \{\gamma (\tau_\varphi (L), L) \mid \varphi \in V \}$$
 So if $C$ is the largest norm of $\Vert H \Vert _{C^0}$ where $H$ belongs to a set of Hamiltonians such that the $\varphi_H$ describe $V$, we get $$\gamma (L) \leq 2C'C$$
 This proves the proposition for $N$ satisfying $(\star)$. Now assume we have a fibration $\pi: P \longrightarrow N$ and $P$ satisfies the Conjecture. Then for $L\in \widehat{\mathfrak L}(T^*N)$ we get 
 $\Lambda_\pi^{-1}(L) \in \widehat{\mathfrak L}(T^*N)$ satisfies  $\gamma-\supp(\Lambda_\pi^{-1}(L))\subset \Lambda_\pi^{-1}(\gamma-\supp(L))  \subset DT^*P$. The first inclusion follows  from Proposition 8.17 in \cite{Viterbo-gammas}, while the second holds provided the projection is a contraction, which we can always assume since $P,N$ are compact. As a result there is some constant $C_P$ such that  $\gamma (\Lambda_\pi^{-1}(L))\leq C_{P}$ but since  according to  Corollary \ref{Cor-5.5} we have $\gamma (\Lambda_\pi^{-1}(L))=\gamma(L)$ (since this is true for $L\in \mathfrak{L}(T^{*}N)$ it will hold in the completion 
 $\widehat{\mathfrak{L}}(T^{*}N)$ this yields a bound on $\gamma (L)$. 
\end{proof} 
 
 \begin{Question}
 What are the values of the best constants in the inequalities of Proposition \ref{Thm-6.3} and Proposition  \ref{Prop-7.11}, 
 \end{Question}
\section{The non-compact case}
 If $N$ is open, we can prove the following easy result (the idea being similar to the $ \varepsilon $-shift trick from \cite{Viterbo-STAGGF},  proposition 4.15, see also \cite{Seyfaddini-surfaces})
\begin{prop}\label{Prop-8.1}  
 Let $N$ be diffeomorphic to $M\times {\mathbb R} $ where $M$ is a closed manifold. There is a constant $C_{N}$ such that for $L\in \mathcal L (T^{*}N)$ satisfying the conditions
  \begin{enumerate} 
 \item  $L=0_{N}$ outside of $T^{*}(M\times [-a,a] )$ 
 \item  $L \subset DT^{*}N$
 \end{enumerate} 
 then  $$ \gamma (L) \leq C_{N}a $$ 
 \end{prop} 
 \begin{proof} 
 Let $\tau_{s}$ be the translation by $s$ in the $ {\mathbb R} $ direction of $N$. Then $L\cap \tau_{s}L \subset 0_{N}$ for $s\geq 2a$ and more precisely the action of any intersection point will be independent from $s$. As a result $FH^{*}(L, \tau_{s}(L))$ is independent from $s$ for $s\geq 2a$ and as $s$ goes to $+\infty$ is equal to 
 $FH^{*}(L,0_{N})\oplus FH^{*}(0_{N},L)$. As a result, for $s\geq 2a$, we have $c_{+}(L,\tau_{s}(L))=c_{+}(L), c_{-}(L,\tau_{s}(L))=c_{-}(L)$. But $\tau_{s}$ is generated by the Hamiltonian $H(x,p)=s\langle p, X(x,p)\rangle$ where $X$ is the vector field $ \frac{\partial}{\partial t}$ on ${\mathbb R} $, thus $H$ is bounded by some constant $C_{N}$ on $L$ (since $ \vert p \vert \leq 1$ on $L$) and according to Lemma \ref{Lemma-7.3} we  have $$\gamma (L, \tau_{s}(L))\leq \gamma (\tau_{s}) \leq \vert s \vert $$
 Thus on one hand $\gamma(L, \tau_{2a}(L)) \leq 2a$  on the other hand $\gamma (L, \tau_{2a}(L)) =c_{+}(L)-c_{-}(L)=\gamma(L)$. As a result
 $$\gamma(L)\leq 2a$$
 \end{proof} 

\section{The spectral distance and locally path-connectedness}
 The spectral distance on $\DHam(M)$ is not - a priori- locally path-connected. Indeed, we do not know whether $\gamma (\varphi) \leq \delta $ implies the existence of a path $\varphi^t$ from $\Id$ to $\varphi$ such that $\gamma (\varphi^t) < \varepsilon (\delta)$ where $ \varepsilon : [0, +\infty[ \longrightarrow [0,+\infty[$ is continuous and vanishes at $0$. One may thus define a different distance
\begin{defn} 
We define the distance $\widetilde \gamma$ on $\DHam (M, \omega)$ as the bi-invariant metric such that
$$\widetilde\gamma(\varphi, \Id)= \widetilde \gamma(\varphi)=\inf\left \{ \sup \left \{ \gamma (\varphi^t) \mid t\in [0,1]\right \}\mid  \varphi^1=\varphi, \varphi^0=\Id\right \}$$
\end{defn} 
This new distance, similarly to the Hofer distance (see \cite{Hofer-distance}), is locally path connected. 
Note that in case $(M, \omega)$ is a Liouville domain\footnote{In principle a Liouville domain is compact. Whenever we say ''compact supported' in a Liouville domain, we mean "compact supported in $M\setminus \partial M)$ !} , we have
\begin{prop} 
If $(M, \omega)$  is a Liouville domain, then for any constant $C$,   the following spaces are path-connected
 \begin{gather*}
  \left \{ \varphi \in \DHam (M, \omega) \mid \gamma (\varphi) \leq C \right \}, 
\left \{ \varphi \in \DHam_c (M, \omega) \mid \gamma (\varphi) \leq C \right \}\\
 \left \{ \varphi \in \widehat{\DHam} (M, \omega) \mid \gamma (\varphi) \leq C \right \}, 
\left \{ \varphi \in \widehat{\DHam}_c (M, \omega) \mid \gamma (\varphi) \leq C \right \}
\end{gather*} 
As a consequence $\DHam (M, \omega), \DHam_{c} (M, \omega), \widehat{\DHam} (M, \omega), \widehat{\DHam}_{c}(M, \omega)$ are locally path-connected so that for $(M, \omega)$ Liouville, we have $\gamma= \widetilde \gamma$.
\end{prop} 
\begin{proof} 
Let $\rho^s$ be the flow of the Liouville vector field, so that $(\rho^s)^*\omega= e^s\omega$. Then $\gamma (\rho^s\varphi\rho_s^{-1})= e^{-s} \gamma (\varphi)$. Since $\DHam (M, \omega) $ and
$\DHam_c (M, \omega)$ are connected by definition, let $\varphi^t$ be a path connecting $\varphi=\varphi^1$ to $\id$. Then let $c(t)=\max \left \{\log(\{ \frac{\gamma (\varphi^t)}{C}), 0\right \})$, so that $c(0)=c(1)=0$. Then 
$\psi^t=\rho^{c(t)}\varphi^t\rho^{-c(t)}$ is a path from $\id$ to $\varphi$ and $\gamma (\psi^t)= e^{-c(t)} \gamma (\varphi^t)\leq C$ and we get a path from $\id$ to $\varphi$ contained in
the $\gamma$-ball of radius $C$. 
\end{proof} 
As a consequence we get
\begin{lem}\label{Lemma-standard-extension}
Let $(\varphi^{t})_{t\in [0,1]}$ be a Hamiltonian isotopy with compact support in the Liouville domain $(M, \omega)$. Then there is a compact supported Hamiltonian isotopy $(\Phi^{t})_{t\in [0,1]}$  in 
$M\times B^{2}(r)$ such that 
\begin{enumerate} 
\item $\Phi^{1} _{\mid M\times \{0\}}= \varphi^{1}$
\item $\gamma (\Phi^{1}) \leq C_{n} \gamma(\varphi^{1})$
\end{enumerate} 
where $C_{n}$ is a constant depending only on the dimension. 
\end{lem} 
\begin{proof} 
The proof is based on the Inverse reduction inequality. First, because $(M, \omega)$  is a Liouville domain, we may assume that $\gamma(\varphi^{t}) \leq \gamma(\varphi^{1})$ for all $t\in [0,1]$. Now let $\chi : [0,r] \longrightarrow [0,1]$ be a smooth map vanishing near $r$, equal to $1$ for $r$ near $0$. 
Let $H(t,q,p)$ be the Hamiltonian generating $\varphi^{t}$. Then set $K(t,z,r,\theta)=\chi(r)H(t\chi(r),q,p)$, where $(r,\theta)$ are polar coordinates in $B^{2}(r)$, so that the flow of $K$ is given, with $z=(q,p)$  by 
$$(q,p,r,\theta) \longmapsto ( Q_{t\chi(r)}(q,p), P_{t\chi(r)}(q,p),  r, \theta +\alpha(t,r,z))$$
where \begin{gather*}  \alpha(t,z,r)=\frac{ \chi'(r)}{r} \int_{0^{t}}H(s\chi(r),Q_{s}(q,p), P_{s}(q,p))ds+\\ \frac{\chi(r)\chi'(r)}{r} \int_{0}^{t}\frac{\partial}{\partial t} H(s\chi(r),Q_{s}(q,p), P_{s}(q,p))ds
\end{gather*} 
And indeed for $r=0$, since $\chi(r)=1, \chi'(r)=0$ we get $$(q,p,0,\theta) \mapsto ( Q_{t}(q,p), P_{t}(q,p),0,\theta)$$
that is $\varphi^{t}$. 
Now let us consider the graph of $\Phi^{t}$ that is 
$$ \left\{(q,P_{t\chi(r)}(q,p), p-P_{t\chi(r)}(q,p), Q_{t\chi(r)}(q,p)-q, \frac{R^{2}}{2}, \theta,  \alpha(t,q,p,r), \frac{r^{2}-R^{2}}{2}) \right \}$$
Notice that here $r=R$ and so the reduction at $r=r_{0}$ is given by 
$$ \left\{(q,P_{t\chi(r)}(q,p), p-P_{t\chi(r)}(q,p), Q_{t\chi(r)}(q,p)-q, \theta, 0) \right \}$$ that is 
$\Gamma (\varphi^{t\chi(r)}\times 0_{S^{1}}$ and $\gamma (\Gamma (\varphi^{t\chi(r)}\times 0_{S^{1}}))= \gamma (\Gamma (\varphi^{t\chi(r)})) = \gamma (\varphi^{t\chi(r)}) \leq \gamma(\varphi^{1})$. 
By the inverse reduction inequality, since in our case $c_+(\Gamma (\phi))$ and $c_+(\Gamma (\Phi))$ are positive and 
$c_-(\Gamma (\phi))$ and $c_-(\Gamma (\Phi))$ are negative, according to Remark \ref{Rems-5.12} (\ref{Rems-5.12-d}), 
$$ \gamma (\Gamma (\Phi^{1})) \leq C_{n}\sup_{r}\gamma (\varphi^{t}) = C_{n} \gamma (\varphi^{1})$$ 
\end{proof}

We could also define a stabilized distance as follows
\begin{defn} 
Let $\varphi \in \DHam (M,\omega)$. A stabilization of $\varphi \in \DHam_c(M)$ is an element $\Phi \in \DHam_c(M\times B^2(1))$  that is  an {\bf extension} of $\varphi$, considered as a map from $M\times \{0\}$ to $M\times \{0\}$,  to $M \times B^2(1)$  and equal to the identity near the boundary $M\times \partial B^2(1)$. In other words $\Phi$ is compact suported in $M\times B^{2}(1)$ and $\Phi_{\mid M\times \{0\}}=\varphi$.
We say that $\Phi$ is a {\bf standard extension} of $\varphi$ If moreover $\Phi$ preserves $M\times S^1(r)$  and  there exists a Hamiltonian isotopy $\varphi^t$ from $\Id$ to $\varphi$ such that 
$$\Phi^t(r,\theta, z)=(r, \Theta (t,r, \theta,z), \varphi^{t\chi(r)}(z))$$ where $\chi(r)=1, \Theta(t,r,\theta,z)=\theta$ for $ r= 1/2$ and $\chi(r) \in C_c^\infty(]0,1[, {\mathbb R} )$ and $0 \leq \chi(r)\leq 1$. In particular  $\Phi^t$ is  the Hamiltonian flow of  $K(t,r,\theta)=\chi(r)H(t\chi(r),z)$. 

\end{defn} 
As a consequence, let $(\varphi^t)_{t\in [0,1]} \in \DHam (M, \omega)$ be a path such that $\gamma (\varphi^t) < \varepsilon$  for all $t\in [0,1]$, so that $\widetilde \gamma(\varphi)< \varepsilon $. Then we claim that $\varphi^t$ has a compact extension to a compact supported  map $\Phi^t $ in $\DHam (M \times B^2(1))$ such that $\gamma(\Phi^t) < \varepsilon $ and $\Phi^t(r,\theta, z)=(r, \Theta (t,r, \theta,z), \varphi^{t\chi(r)}(z))$ where $\chi(r)=1, \Theta(t,r,\theta,z)=\theta$ for $ r= 1/2$ and $\chi(r) \in C^\infty(]0,1[, {\mathbb R} )$ and $0 \leq \chi(r)\leq 1$. This is easy, by taking $\Phi^t$ to be the Hamiltonian flow of  $K(t,r,\theta)=\chi(r)H(t\chi(r),z)$. Indeed its flow is given by $$(r,\theta, z) \mapsto (r, \Theta (t,r, \theta,z), \varphi^{t\chi(r)}(z))$$ so consider the image $\Gamma (\Phi)$ of  its graph 
by 
$$ (r,\theta,z, R, \Theta, Z) \mapsto (r, R-r,  \theta-\Theta, \Theta, z,Z)=(u,v,U,V,z,Z)$$ 
with symplectic form $dU\wedge du + dV\wedge dv +z^*\omega_M-Z^*\omega_M$
\begin{prop} \label{Prop-3.3}
We have for any extension $(\Phi^t)_{t\in [0,1]}$ of $(\varphi^t)_{t\in [0,1]}$ the inequality $$\sup_{\tau \in [0,t]}\gamma (\varphi^\tau)\leq \gamma (\Phi^t)$$ 
If moreover $\Phi$ is  a standard extension, then  $$\gamma(\Phi^t) \leq C_d \sup_{\tau \in [0,t]}\gamma (\varphi^\tau)$$
\end{prop} 
\begin{proof} The first is just the standard reduction inequality, since one of the reductions of $\Gamma(\Phi)$ is $\Gamma (\varphi)$. 
For the second one, if $\Phi$ is  standard extension, the reduction of $\Gamma (\Phi)$ by $r=r_0$ is the graph of $\varphi^{t\chi(r_0)}$. In particular for $r_0$ it is the graph of $\varphi$. 
This is an embedded Lagrangian, hence we may apply  both the direct and reverse 
reduction inequality to get the Proposition. 
\end{proof} 

\begin{prop}\label{Prop-9.4} Let $(M, \omega)$ be an aspherical symplectic manifold. 
\begin{enumerate} 
\item 
We have $\gamma (\varphi)\leq \widetilde \gamma (\varphi)$.
\item Let  $\varphi \in \DHam (M, \omega_M)$ Then for a standard extension $\Phi$ of $\varphi$ to $\DHam(M\times D^2(0,1), \omega_{M}\oplus \sigma)$. Then 
$$ \widetilde\gamma (\varphi) \leq \gamma (\Phi) \leq C_d  \widetilde\gamma (\varphi)  $$ where $d= \codim (M)$. 
\end{enumerate} 
\end{prop} 
\begin{proof} 
Obvious from the definition.
\end{proof} 

In particular if we set
\begin{defn} Define $\widehat \gamma (\varphi)$ to be the infimum of all $\gamma (\Phi)$ for all extensions of $\varphi$.
\end{defn} 
The previous Proposition implies
\begin{prop} We have the inequalities $$\gamma (\varphi) \leq \widehat \gamma (\varphi) \leq \widetilde \gamma (\varphi)$$
\end{prop} 
\begin{cor} 
Let $(M, \omega) \longrightarrow (P, \omega_{P})$ be a symplectic embedding. Then there exists a non-expanding embedding $(\DHam_{c}(M,\omega_{M}), \widetilde \gamma) \longrightarrow (\DHam_{c}(P,\omega_{P}), \widetilde \gamma)$ and therefore of their completions. 
\end{cor} 
\begin{proof} 
Indeed, associate to  $\varphi$ its standard extension $\Phi$. Proposition \ref{Prop-9.4} implies that $\widetilde {\gamma}(\Phi) \leq \widetilde {\gamma}(\varphi)$, hence the statement. 
\end{proof} 

\appendix
\section{ The Kislev-Shelukhin inequality}\label{Appendix1}\index{Kislev-Shelukhin inequality}\label{Appendix-KS-inequality}
The goal of this appendix is to make the paper more self-contained by giving a proof of the Kislev-Shelukhin inequality. There is nothing new here compared to \cite{Ki-Sh} except that we reduced their proof to an abstract result on persistence modules endowed with a ring structure, but this is already implicit in their work.  

Let $\beta(L_1,L_2)$ be the boundary depth of $FH^*(L_1, L_2)$. This is the size of the longest bar of the barcode associated to the persistence module $t \mapsto FH^*(L_1, L_2;t)$. It is also twice the bottleneck distance between this persistence module and the persistence module $0$. Then we have 
\begin{prop} (\cite{Ki-Sh})\label{Ki-Sh}
We have the inequality
$$ \beta (L_1,L_2) \leq \gamma (L_1,L_2)$$
 \end{prop} 
 
 Given Lagrangians, $(L_i)_{i\in I}$, we consider the persistent modules $V^{i,j}_t=FH_*(L_i,L_j, t)$ for $i, j$ in some set $I$, $t\in {\mathbb R} $, endowed with the following structures:
 
 \begin{enumerate} 
 \item \label{Property-1}The persistent module structure that is a family of maps $r_{s,t}^{i,j}: V^{i,j}_s \longrightarrow  V^{i,j}_t$ for $s\leq t$ satisfying $r_{t,u}\circ r_{s,t}= r_{s,u}$ and $r_{t,t}=\Id$. 
 We shall only consider the case where $V^{i,j}_t=0$ for $t<<0$ and $V^{i,j}_t=V^{i,j}_{\infty}$ for $t>>0$. We also write $r_a^{i,j}$ for one of the maps $r_{t,t+a}^{i,j}$ when $t$ is unspecified. 
 \item \label{Property-2} A composition map induced by the triangle product $$\mu_{i,j,k}: V^{i,j}_s\otimes V^{j,k}_t \longrightarrow V^{i,k}_{s+t}$$ 
 \item \label{Property-3} A PSS\footnote{the existence is originally due to \cite{P-S-S}} unit $u_{i,j}\in V^{i,j}$ such that the maps $$\mu_{i,j,k}(\bullet \otimes u_{j,k}): V^{i,j}_\infty \longrightarrow V^{i,k}_{\infty}$$  and 
 $$\mu_{i,j,k}(u_{i,j} \otimes \bullet ): V^{i,k}_\infty \longrightarrow V^{j,k}_{\infty}$$
are isomorphisms and $\mu_{i,j,k}(u_{i,j}\otimes u_{j,k})=u_{i,k}$. 
\end{enumerate} 
satisfying the following conditions
\begin{enumerate} 
  \item \label{Cond-1} This composition is associative, i.e. $$\mu_{i,k,l}\circ (\mu_{i,j,k}\otimes \id_{k,l}) =\mu_{i,j,l}\circ(\id_{i,j}\otimes\mu_{j,k,l})$$
 \item  \label{Cond-2} For $i=j=k$ this defines a ring structure on $V^i_\infty=V^{i,i}_\infty$. In particular $V^i_\infty$ is a unitary ring (the unit is $u_i$) with spectrum concentrated at $t=0$ : the vector space $V^i_t=V^{i,i}_t$ is zero for $t<0$ and equal to $V^i=V^i_\infty$ for $t>0$.
 
\item  \label{Cond-3}  The $\mu_{i,i,j}$ (resp. $\mu_{i,j,j}$) defines a structure of left $V^i$-module (resp. right $V^{j}$-module) on $V^{i,j}$ and $\mu_{i,j,k}$ is a morphism of $(V^i,V^k)$-modules.
\item  \label{Cond-4} for $a>0$ the maps
 $$\mu(u_i\otimes \bullet) : V^{i,j}_t \longrightarrow V^{i,j}_{t+a}$$ and 
 $$\mu(\bullet \otimes u_j) : V^{i,j}_t \longrightarrow V^{i,j}_{t+a}$$
 coincide with the restriction map $r_{t,t+a}$.  
 \end{enumerate} 
 
 With these structures 
 \begin{defn} We set
 \begin{enumerate} 
 \item $0$ is the persistence module such that $0_t=0$. 
 \item If $V$ is a persistence module we denote by $V[a]$ the persistence module defined by $(V[a])_t=V_{t+a}$ and by $\varphi_a: V \longrightarrow V[a]$
 the map induced by $r_{t,t+a}$. 
 \item The interleaving distance $d_I (V^{i,j}, V^{k,l})$ between $V^{i,j}$ and $V^{k,l}$ is the infimum of the set of $\frac{a}{2} \in {\mathbb R} $ such that there exists
 interleaving morphisms $\sigma: V^{i,j} \longrightarrow V^{k,l}[a']$ and $\tau: V^{k,l} \longrightarrow V^{i,j}[a'']$ such that for $a=a'+a''$ we have $\sigma\circ \tau =r^{k,l}_a$ and 
 $\tau\circ \sigma = r^{i,j}_{a}$
 \item $\beta (V^{i,j})=\beta (V^{i,j}, V^{i,j}_\infty)$ is the boundary depth, equal to $2d_I(V^{i,j})=2d_I (V^{i,j}, V^{i,j}_\infty)$
\item We define $c_{i,j}=c(u_{i,j},V^{i,j})= \inf \{ c \in {\mathbb R} \mid u_{i,j}\in V^{i,j}_c\}$ and $\gamma(V^{i,j})=c_{i,j}+c_{j,i}$
 \end{enumerate} \end{defn} 
 From now on we omit the $\mu$ and denote the product by ``$\cdot $ '', i.e. $$x\cdot y=\mu_{i,j,k}(x\otimes y)$$ 
 Since $u_{i,j} \in V^{i,j}_{c_{i,j}}$ we get a map $V^{i,j} \longrightarrow V^{i,j}[c_{i,j}]$
 \begin{prop} \label{Prop-App1}
 We have $\beta (V^{i,j}, V^{i,k}) \leq \gamma (V^{j,k})$
 \end{prop} 
 \begin{proof} We shall omit the index on $\mu$ when obvious. 
 We have $$u_{k}= u_{k,j}\cdot u_{j,k}$$  $$\mu ( \bullet \otimes u_{j,k} ) : V^{i,j} \longrightarrow V^{i,k}[c_{k,j}]$$ and $$\mu(\bullet \otimes u_{k,j}):  V^{i, k} \longrightarrow V^{i,j}[c_{k,j}]$$
The composition is 
 $$\mu(\mu( \bullet \otimes u_{k,j}) \otimes u_{j,k}) : V^{i,j} \longrightarrow V^{i,k}[c_{i,j}]  \longrightarrow V^{i,j}[c_{i,j}+c_{j,i}]$$
 By associativity, this is equal to $\mu(\bullet \otimes u_k )$ but by Condition \ref{Cond-4} this is just the map $r^{i,j}_{c_{j,k}+c_{k,j}}: V^{i,j} \longrightarrow V^{i,j}[c_{j,k}+c_{k,j}]$. 
 As a result $\mu ( \bullet \otimes u_{j,k} )$ and $\mu(\bullet \otimes u_{k,j})$ define an interleaving between $V^{i,j}$ and $V^{i,k}$, so their interleaving distance is bounded by $c_{j,k}+c_{k,j}=\gamma (V^{j,k})$. 
 \end{proof} 
 
 Now let $V$ be a persistence module. Consider the map $V_t \longrightarrow V_\infty$ and set $V^\infty_t= \Image (r_{t,\infty})$. This defines a persistence module by 
 $r^\infty_{s,t}: V^\infty_s \longrightarrow V^\infty_t$ sending $r_{s,\infty}(x)$ to $r_{t,\infty}(x)$ by $r_{s,t}$. There is an obvious map $V \longrightarrow V^\infty$ induced by $r_{t,\infty}$. In terms of barcode, $V^\infty$ is obtained from $V$ by deleting all finite bars. 
 
  \begin{cor}\label{Cor-A4}
  We have $$\beta (V^{i,j}) \leq \gamma (V^{i,j})$$
  \end{cor} 
  \begin{proof} 
  We have the persistence modules $V, V^\infty$. Since there is a map $V \longrightarrow V^\infty$, we have a product map 
  $V \otimes V^\infty \longrightarrow V^\infty$ so setting $V^{1,2}=V, V^{2,3}=V^\infty, V^{1,3}=V^\infty$, the Proposition implies $$\beta (V)=\beta (V, V^\infty) \leq \gamma (V)$$
  \end{proof} 
\begin{proof}[Proof of Proposition \ref{Ki-Sh}] This follows by applying Corollary \ref{Cor-A4} to the persistence module $FH^*(L_1,L_2,t)$.
\end{proof}    
 \begin{rems}
 \leavevmode
    \begin{enumerate} 
    
    \item  This trivially holds for example for the persistence module associated to sublevel sets of a function on a compact manifold. Then in the above Corollary we can have almost equality in the non-degenerate case, and equality in general. Indeed, let $f$ be a function on $S^1$ with two local maxima near level $1$ and two local minima near level $0$ and such that 
    $\sup_{\theta\in S^1} f(\theta)=1, \inf_{\theta\in S^1} f(\theta)=0$. Then $\gamma(f)=1$, and since all bars must connect a point of index $1$ to a point of index $0$, and there is a non trivial bar, the bar must have length one, and $\beta(f)=1- \varepsilon $ while $\gamma(f)=1$.
    \item Note that if $L_1, L_2$ have generating functions quadratic at infinity, $S_1, S_2$ we could try to apply this to the persistence module $H^*((S_1\ominus S_2)^t,(S_1\ominus S_2)^{-\infty})$. However the Kislev-Shelukhin inequality does not hold if $L_1$ or $L_2$ is not embedded !  For example if $f$ is quadratic at infinity on $ {\mathbb R}$ (i.e. $N=\{pt\}$), we have $\gamma (f)=0$ (because $1_N=\mu_N$ on $H^*(pt)$), but $\beta (f)$ can be large (for example if $f(x)=x^4-2x^2$, $\beta (f)=1$ and even though $f$ is not quadratic at infinity, it is convex, so can be deformed to $x^2$ at infinity without adding critical points)

    \item 
  Here we considered the case of coefficients in some field. The case of a ring coefficients can be adapted as follows. First $\beta(L_1,L_2)$ must be defined using the interleaving distance. 
  Second we must define $\gamma (L_1,L_2)$ as $c(\mu,\overline L)+c(\mu, L)$. 
  \end{enumerate} 
  \end{rems} 
\printbibliography
\end{document}